\documentclass[12pt]{article}
\usepackage{graphicx}
\usepackage{amssymb}
\usepackage{amsmath}
\usepackage{amsthm}
\usepackage{mathptmx}
 \usepackage{latexsym}

\usepackage{tikz}
\usetikzlibrary{decorations.pathreplacing,calc}

\usepackage{here}
\usepackage{cancel} 
\usepackage[]{hyperref}

\def\<#1>{\langle#1\rangle}

\newtheorem{thm}{Theorem}[section]
  \newtheorem{definition}[thm]{Definition}
   \newtheorem{corol}[thm]{Corollary}
 \newtheorem{remark}[thm]{Remark}

   \theoremstyle{definition}
\newtheorem{rem}[thm]{Remark}

\begin{document}

\title{The arithmetic of triangles  }
\author{Edward Mieczkowski\footnote{
 Gdynia, Poland;  E-mail: \texttt{edmieczyk@gmail.com}.}}

\date{}

\maketitle

\begin{abstract}

\noindent In this paper, we consider a set of similar triangles with parallel sides, along with a set of points in the plane.
It turns out that the set\\
$\mathbb{R}_2= \{\pm  \<x >=\pm (x^2,x,1); x\in\mathbb{R} \}$\\
describes this set of triangles quite well. The set $\mathbb{R}_2$ is a subset of the ring\\
$\mathbb{R}^3=\mathbb{R}\times\mathbb{R}\times\mathbb{R}= \{ (x,y,z) ; \; x,y,z\in\mathbb{R} \}$ \\
with addition and multiplication defined coordinate-wise.

\noindent The set $\mathbb{R}_2$ is equipped with two operations. Multiplication is inherited from the ring $\mathbb{R}^3$, while addition is a ternary operation that represents homothety and translation of elements in $\mathbb{R}_2$. Depending on the relative sizes of the arguments in the addition operation, there are six distinct types of geometric interpretation.

\noindent However, the defined addition has its limitations. It turns out that, within this framework, the reduction of terms with different signs is not always possible. This leads to the distinction between an equation that is true in the arithmetic sense and one that is true in the geometric sense.

\noindent A novel form of addition in $\mathbb{R}_2$ leads to intriguing properties of multiplication in $\mathbb{R}_2$, which are examined in a dedicated chapter.

\noindent In the  next section  we use the construction of adding  to  describe  the dissection of the triangle  into 15 triangles of different sides.
 
\noindent In the final two sections, we consider a set of two kinds of vectors, along with a set of points on the line. The set\\
$\mathbb{R}_1= \{\pm  \<x >=\pm (x,1); x\in\mathbb{R} \}$\\
describes this set vectors quite well and it is a one-dimensional reduction of the set $\mathbb{R}_2$. 

 \end{abstract}

Math. Subj. Class.: 03H15, 05B45, 11U10, 52C20.

\newpage
\subsection*{ Introduction}

Let us take the ring $\mathbb{R}^3=\mathbb{R}\times\mathbb{R}\times\mathbb{R}= \{ (x,y,z) ; \; x,y,z\in\mathbb{R} \}$ with  additon and multiplication\\
\begin{align}
(x_1,y_1,z_1)+(x_2,y_2,z_2)=&\ (x_1+x_2,y_1+y_2,z_1+z_2), \label{1} \\
(x_1,y_1,z_1)\cdot (x_2,y_2,z_2)=&\ (x_1\cdot x_2,y_1\cdot y_2,z_1\cdot z_2). \label{2}
\end{align}
Let us consider the subset of the ring $\mathbb{R}^3$, the set $\mathbb{R}_2= \{\pm  \<x >=\pm (x^2,x,1); x\in\mathbb{R} \}$.\\
It is closed under multiplication (\ref{2}) but not under addition (\ref{1}).\\
The set $\mathbb{R}_2$ is closed under    the following kind of addition 
(it should be noted that the summands in below Eq.(\ref{3}) are not  the coordinates of the ring $\mathbb{R}^{3})$.
  \begin{align}
\forall x,y,z,t \in \mathbb{R}\qquad
       \<x+y+z+t> = &\  \< x+y+t > + \< x+z+t> + \< y+z+t >  \nonumber \\
                    & - \< x+t> - \<y+t> - \<z+t> +\<t> \qquad
    \label{3}
\end{align}
because the equations
\begin{align*}
\forall x,y,z,t \in \mathbb{R}\quad \forall i=2,1,0\qquad \qquad \\
       (x+y+z+t)^i = &\  ( x+y+t )^i + ( x+z+t)^i + ( y+z+t )^i   \\
                    & - ( x+t)^i - (y+t)^i - (z+t)^i +(t)^i
   \end{align*}
are true.\\
If we multiply the Eq. (\ref{3}) by $-1$, we get the definition of addition for the elements $-\<x>$ of the set $\mathbb{R}_2$.

\noindent Using Eq.(\ref{3}), one can derive by induction the identity
 \begin{align}
\forall n \in \mathbb{Z}\qquad
       \<n>=\frac{n(n+1)}{2}\<1>-(n-1)(n+1)\<0> + \frac{n(n-1)}{2}\<-1>.
    \label{4}
\end{align}
But you can check purely for accounting that  Eq. (\ref{4}) holds for each $x \in \mathbb{R}$.
\begin{align}
\forall x \in \mathbb{R}\qquad
       \<x>=\frac{x(x+1)}{2}\<1>-(x-1)(x+1)\<0> + \frac{x(x-1)}{2}\<-1>.
    \label{5}
\end{align}
Let us transform Eq. (\ref{5}). 
\begin{align*}
\< x > = & \ \frac{x^2+x}{2}\< 1> -(x^2-1)\<0> + \frac{x^2-x}{2}\< -1> \\
		= &\ x^2\frac{\< 1> -2\<0>+\< -1>}{2} +x\frac{\< 1> - \< -1>}{2} + 1\<0>.
\end{align*}
 It is easy to check that 
 the elements  $\frac{\< 1> -2\<0>+\< -1>}{2} = A_2$,  $\frac{\< 1> - \< -1>}{2} = A_1$, $\<0>=A_0$ are orthogonal. So we can put
$ A_2=(1,0,0)$, \; $A_1=(0,1,0)$, \; $A_0=(0,0,1)$ and now we know why $\<x>=(x^2,x,1)$.\\

\section{Geometric interpretation of the set $\mathbb{R}_2$.} 

\noindent  Let us set any closed  triangle on the plane (the triangle does not have to be equilateral) and let us denote it by the symbol $\< 1>$. Each triangle obtained from the fixed trinagle $\<1>$ by any translation will be denoted by $\<1>$ too.\\
The triangle obtained from the triangle $\<1>$ by any homothety of ratio $x>0$ \cite{Wol} will be denoted by $\<x>$.\\
Each point of the plane will be denoted by $\<0>$.\\
The symbol  $-\< x>$ denotes the triangle, which lying on the triangle $\< x>$ gives an empty set  denoted by an element $(0,0,0) \in \mathbb{R}^3$.\\
An element $\<x>$ for each $x \in \mathbb{R},\: x>0$  will be interpreted as a closed triangle and marked with a black color 
\tikz{\draw[fill,gray!70] (0,0)--(1,0)--(.5,.8)--(0,0);
\draw (0,0)--(1,0)--(.5,.8)--(0,0);
\draw[fill](0,0)circle(1pt);\draw[fill](1,0)circle(1pt);
\draw[fill](.5,.8)circle(1pt);} \ .
 The interior of this
triangle should be black but we would like to mark that sides and vertexes belong to  this triangle,  so we coloured  it in gray.\\
An element $-\<x>$ will be interpreted as a closed triangle (we will  call it sometimes antytriangle) and marked with a red color
\tikz{\draw[fill,red!30] (0,0)--(1,0)--(.5,.8)--(0,0);
\draw[red] (0,0)--(1,0)--(.5,.8)--(0,0);
\draw[fill,red](0,0)circle(1pt);\draw[fill,red](1,0)circle(1pt);
\draw[fill,red](.5,.8)circle(1pt);} \;.\\
An element $\<0>$ will be marked as a black point and an element $-\<0>$ will be marked as a red point.\\
To see the result of putting the black triangle on the red one, we will denote the empty set of green.\\
An element $\<-x>$ for each $x \in \mathbb{R},\: x>0$  will be interpreted as a open triangle symmetrical to the  triangle  $\<x >$ in relation to any side and will be marked with a black color
\tikz{\draw[fill,gray!70] (.5,0)--(1,.8)--(0,.8)--(.5,0);
\draw[green] (.5,0)--(1,.8)--(0,.8)--(.5,0);
\draw[fill,green](.5,0)circle(1pt);\draw[fill,green](1,.8)circle(1pt);
\draw[fill,green](0,.8)circle(1pt);} .\\
At the end an element $-\<-x>$ for each $x \in \mathbb{R},\: x>0$  will be interpreted as a open triangle symmetrical to the  triangle  $-\<x >$ in relation to any side and will be marked with a red color
\tikz{\draw[fill,red!30] (.5,0)--(1,.8)--(0,.8)--(.5,0);
\draw[green] (.5,0)--(1,.8)--(0,.8)--(.5,0);
\draw[fill,green](.5,0)circle(1pt);\draw[fill,green](1,.8)circle(1pt);
\draw[fill,green](0,.8)circle(1pt);} .\\

\noindent Why the triangles $\<-x>$ and $-\<-x>$ for $ x>0$  are open we will see after introducing the geometric interpretation of the operation (\ref{3}).\\

\noindent The geometric interpretation of the individual components of Eq. (\ref{4}) is as follows (Fig. \ref{f1}).
\begin{figure}[htbp]
\begin{center} 
\begin{tikzpicture}[>=stealth]
\draw[fill,gray!60] (0,0)--(3,4.8)--(6,0)--(0,0);
\draw (0,0)--(2,0)--(1,1.6)--(0,0);
\draw (1,0)--(3,0)--(2,1.6)--(1,0);
\draw (.5,.8)--(2.5,.8)--(1.5,2.4)--(.5,.8);
\draw (.5,.8)--(1,0)--(2,0)--(2.5,.8)--(2,1.6)--(1,1.6);
\draw[dashed] (1.8,2.88)--(1.5,2.4)--(2.1,2.4);
\draw[dashed] (2.3,2.08)--(2,1.6)--(2.6,1.6);
\draw[dashed] (2.8,1.28)--(2.5,.8)--(3.1,.8);
\draw[dashed] (3.3,.48)--(3,0)--(3.6,0);
\draw (5,0)--(6,0)--(3,4.8)--(2.5,4)--(5,0);
\draw (5,0)--(5.5,.8)--(4.5,.8)--(5,1.6)--(4,1.6)--(4.5,2.4)
--(3.5,2.4)--(4,3.2)--(3,3.2)--(3.5,4)--(2.5,4)--(3,4.8);
\draw[dashed](2.2,3.52)--(2.5,4); \draw[dashed](4.4,0)--(5,0);
\draw[dashed](2.4,3.2)--(3,3.2)--(2.7,2.72);
\draw[dashed](2.9,2.4)--(3.5,2.4)--(3.2,1.92);
\draw[dashed](3.4,1.6)--(4,1.6)--(3.7,1.12);
\draw[dashed](3.9,.8)--(4.5,.8)--(4.2,.32);
\draw node[] at (.5,.25){\small{$\<1>$}};
\draw node[] at (1.5,.25){\small{$\<1>$}};
\draw node[] at (2.5,.25){\small{$\<1>$}};
\draw node[] at (1,1.05){\small{$\<1>$}};
\draw node[] at (1,.55){\small{$\<\text{-}1>$}};
\draw node[] at (2,.55){\small{$\<\text{-}1>$}};
 \end{tikzpicture}
\caption{} \label{f1}
\end{center}
\end{figure}
\noindent The triangle $\<n>$ has $n^2$  of all triangles $\<1>$ and $\<-1>$.\\
If we assume that the length of each side of the triangle $\< 1> $ is equal to $1$ (even when the triangle is not equilateral) then $n$ is the length of each side of the triangle $\<n>$.\\
 The number 1 in $(n^2,n,1)$ means one  triangle. (Each point $\<0>$ is a triangle with sides equal to 0. Each triangle $\<-n>$ has sides equal to -n).  \\
The triangle $\<n>$ has $\frac{n(n+1)}{2}$ of triangles $\<1>$ and $\frac{n(n-1)}{2}$ of  triangles $\<-1>$.\\
Since the vertices of the closed triangles $\<1>$ overlap, these vertices as points $\<0>=(0,0,1)$ must be subtracted  in the description (\ref{4}) of the triangle $\<n>$.\\
We have $3(n-1)$  double vertices lying on the 3 sides of the triangle and $\frac{(n-2)(n-1)}{2}$ triple vertices lying inside the triangle $\<n>$. So we must substract \\[.1cm]
$3(n-1)+2\dfrac{(n-2)(n-1)}{2}= (n-1)(n+1)$\\[.1cm]
vertices $\<0>$ lying in the description (\ref{4}) of the triangle $\<n>$.\\

\noindent Using  linear combinations of  elements of the set $\mathbb{R}_2$ we can describe different geometric figures (as elements of the ring $\mathbb{R}^3$) and thus better understand the meaning of the  number $x$ in the element $\<x^2,x,1>$. (The colors white and green mean the same, ie. empty set) (Fig.~\ref{f2}, ~\ref{f3}).
\begin{figure}[htbp]
\begin{center} 
\begin{tikzpicture}[>=stealth,scale=0.50]
\draw[fill,gray!60](2.2,5)--(0,1)--(6.6,1)--(4.4,5)--cycle;
\draw[dashed,green](2.2,5)--(3.3,7)--(4.4,5);
\draw[green,thick](2.2,5)--(4.4,5);
\draw(2.2,5)--(0,1)--(6.6,1)--(4.4,5);
\draw[fill,green](2.2,5)circle(1.5pt);\draw[fill,green](4.4,5)circle(1.5pt);
\draw node[] at (3.3,1.3){\footnotesize{$a$}};\draw node[] at (3.3,4.5){\footnotesize{$b$}};
\draw node[right] at (5.6,3){\footnotesize $a-b$};
\draw node at (3.3,0){\footnotesize $\< a > - \<b > = (a^2-b^2,a-b,0), $};
\draw node at (3.3,-1){\footnotesize $ a>b>0$.};

\draw[fill,gray!60](12,7)--(14.2,3)--(16.4,3)--(18.6,7)--cycle;
\draw[green,thick](14.2,3)--(16.4,3);
\draw(16.4,3)--(18.6,7)--(12,7)--(14.2,3);
\draw[fill,green](14.2,3)circle(1.5pt);\draw[fill,green](16.4,3)circle(1.5pt);
\draw node[] at (15.3,3.4){\footnotesize{$-b$}};\draw node[] at (15.3,6.6){\footnotesize{$-a$}};
\draw node[right] at (17.5,5){\footnotesize $-a-(-b)$};
\draw[dashed,green](14.2,3)--(15.3,1)--(16.4,3);
\draw node at (15.3,0){\footnotesize $\< -a > - \<-b > = (a^2-b^2,-a+b,0), $};
\draw node at (15.3,-1){\footnotesize $ a>b>0$.};
 \end{tikzpicture}
\caption{} \label{f2}
\end{center}
\end{figure}
\begin{figure}[htbp]
\begin{center} 
\begin{tikzpicture}[>=stealth,scale=0.50]
\draw[dashed,green](0,1)--(3.3,7)--(6.6,1)--cycle;
\draw[fill,gray!60](0,1)--(2.2,1)--(4.4,5)--(2.2,5)--cycle;
\draw(2.2,5)--(0,1)--(2.2,1);
\draw[green,thick](2.2,1)--(4.4,5)--(2.2,5);
\draw[fill,green](2.2,1)circle(1.5pt);\draw[fill,green](2.2,5)circle(1.5pt);
\draw[fill,red](4.4,5)circle(1.5pt);
\draw node[] at (3.6,3){\footnotesize{$a$}};\draw node[] at (3.3,5.4){\footnotesize{$b$}};
\draw node[] at (3.3,0){\footnotesize $\< a+b > -\<a>- \<b > = (2ab,0,-1), $};
\draw node at (3.3,-1){\footnotesize $ a,b>0$.};

\draw[dashed,green](12,1)--(15.3,7)--(18.6,1)--cycle;
\draw[fill,gray!60](13.1,1)--(16.4,1)--(17.5,3)--(16.4,5)--(14.2,5)--(12.55,2)--cycle;
\draw(13.1,1)--(16.4,1)--(17.5,3)--(16.4,5)--(14.2,5)--(12.55,2)--cycle;
\draw[green,thick](16.4,1)--(17.5,3);
\draw[green,thick](16.4,5)--(14.2,5);
\draw[green,thick](12.55,2)--(13.1,1);
\draw[fill,green](16.4,1)circle(1.5pt);\draw[fill,green](17.5,3)circle(1.5pt);
\draw[fill,green](16.4,5)circle(1.5pt);\draw[fill,green](14.2,5)circle(1.5pt);
\draw[fill,green](12.55,2)circle(1.5pt);\draw[fill,green](13.1,1)circle(1.5pt);
\draw node[] at (14.8,1.3){\footnotesize{$b+t$}};\draw node[] at (15.3,5.4){\footnotesize{$b$}};
\draw node[] at (12.6,3.5){\footnotesize{$a+t$}};\draw node[] at (17.2,1.9){\footnotesize{$a$}};
\draw node[] at (17.8,4.1){\footnotesize{$c+t$}};\draw node[] at (12.6,1.3){\footnotesize{$c$}};
\draw node at (15.3,0){\footnotesize $\< a+b+c+t > - \<a >-\<b>-\<c> = (x,t,-2), $};
\draw node at (15.3,-1){\footnotesize $ a,b,c>0$.};
\end{tikzpicture}
\caption{} \label{f3}
\end{center}
\end{figure}

\noindent\textit{Geometric interpretation of the operation (\ref{3}).}\\ 

\noindent Let us fix the ordered successive components   $ x,y,z $ of the sum $ \< x+y+z+ t > $. This components extend or shorten the triangle $ \< t> $ in the directions I, II, III, or \mbox{I ', II ', III'}, depending on whether the numbers $ x, y, z $     are positive or negative (Fig.~\ref{f4}).

\noindent Therefore Eq. (\ref {3}) should be properly written as
\begin{align*}\begin{split}
  \forall x,y,z,t \in \mathbb{R} \qquad\quad \\
   \<x+y+z+t>&= \< x+y+0+t> + \<x+0+z+t> + \< 0+y+z+t>  \\
             &\quad - \<x+0+0+t> -\< 0+y+0+t> - \< 0+0+z+t > \\
             &\quad + \< 0+0+0+t >.
     \end{split}   \end{align*}

\noindent Since the above record is long and inconvenient, we will replace it with the following one

\begin{align}\begin{split}
  \forall x,y,z,t \in \mathbb{R} \qquad\quad \\
   \<x,y,z,t>&= \< x,y,0,t> + \<x,0,z,t> + \< 0,y,z,t>  \\
             &\quad - \<x,0,0,t> -\< 0,y,0,t> - \< 0,0,z,t > 
              + \< 0,0,0,t >,
    \label{7}
 \end{split}   \end{align}
 where we can write $\< 0,0,0,t >$   as $\< t >$.\\
 From now on, we will replace Eq. (\ref{3}) with Eq. (\ref{7}).\\
 But sometimes, especially when referring to a single triangle, we will write briefly as  $\<x+y+t>, \  \<x+t>, \ \<t>$. 

\begin{figure}[htbp]
\begin{center} 
\begin{tikzpicture}[>=stealth]
\draw[fill,gray!60] (0,0)--(3.6,0)--(2.3,2.08)--(.5,-.8)--(3.1,-.8)--(1.3,2.08)--(0,0);
\draw (0,0)--(3.6,0)--(2.3,2.08)--(.5,-.8)--(3.1,-.8)--(1.3,2.08)--(0,0);
\draw node[] at (1.8,.4){\small{$\<t>$}};
\draw node at (3,.2){\footnotesize $ x$};\draw node at (2.2,1.6){\footnotesize $ x$};
\draw node at (.5,.2){\footnotesize $ y$};\draw node at (1.4,1.6){\footnotesize $ y$};
\draw node at (.9,-.4){\footnotesize $ z$};\draw node at (2.65,-.4){\footnotesize $ z$};
\draw [decorate,decoration={brace,amplitude=5pt},
       xshift=-4pt,yshift=2pt]
   (0,0) -- (1.3,2.08); 
\draw node at (-.3,1.3){\footnotesize $\<0,y,0,t>$};\draw node at (-1,1.8){\Large II}; 
\draw [decorate,decoration={brace,amplitude=5pt},
       xshift=0pt,yshift=-3pt]
   (3.1,-.8)--(.5,-.8);
\draw node at (1.8,-1.3){\footnotesize $\<0,0,z,t>$};\draw node at (1.8,-1.8){\Large III}; 
\draw [decorate,decoration={brace,amplitude=5pt},
       xshift=4pt,yshift=2pt]
   (2.3,2.08)--(3.6,0); 
\draw node at (3.9,1.3){\footnotesize $\<x,0,0,t>$};\draw node at (4.4,1.8){\Large I}; 
\draw node at (3.6,-.5){\Large II'}; \draw node at (.1,-.5){\Large I'};
\draw node at (1.85,2.4){\Large III'};  
 \end{tikzpicture}
\caption{Directions of  changes of the triangle $\< t>$
demand  on components $x, y, z$.} \label{f4}
\end{center}
\end{figure}

\noindent Below we are further examples of the creation of new triangles from the triangle $ \< t > $ (Fig.~\ref{f5}-\ref{f7}).\\
Because some components of Eq. (\ref{7}) are positive (i.e. black) and others are negative (i.e. red),  for simplicity we  omit these colors in Fig. ~\ref{f5}-\ref{f7}.\\
The triangle $ \<t> $ has vertices $ ABC $, and  the triangle $ \<x,y,z,t>$  has vertices $ A'B'C'$.\\
The side $B'C'$ of the triangle $\<x,0,0,t>$ is created by a moving the side $BC$ in the  direction I or I' while $A=A'$ and the sides $A'B'$ and $A'C'$  are lying on the lines containing respectively the sides $AB$ and $AC$ (Fig.~\ref{f5}).\\
The sides $B'C'$ and $A'C'$ of the triangle $\<x,y,0,t>$ are created by a moving respectively  the sides $BC$ and $AC$  while the side $A'B'$  is lying on the line containing the side $AB$ (Fig.~\ref{f6}).\\
The sides created by moving the sides $BC$, $AC$ and $AB$ by values $x,\ y$  and $z$ are blue, yellow and brown respectively.
\begin{figure}[htbp]
\begin{center} 
\begin{tikzpicture}[>=stealth,scale=0.40]

\draw[](4.3,8)--(1,2)--(7.6,2);
\draw[blue,thick](4.3,8)--(7.6,2);
\draw (3.2,6)--(5.4,2);
\draw[->](4.5,3.9)--(5.9,4.8);
\draw node at (4.8,4.6){\footnotesize $x$};
\draw node at (2.9,3.3){\footnotesize $\< t>$};
\draw node at (.4,1.5){\footnotesize $A'=A$};\draw node at (5.3,1.5){\footnotesize $B$};
\draw node at (2.7,6.1){\footnotesize $C$};
\draw node at (7.6,1.5){\footnotesize $B'$};\draw node at (3.9,8.1){\footnotesize $C'$};
\draw node at (3.8,0){\footnotesize $x,\;t>0$};

\draw(12.5,2)--(16.9,2)--(14.7,6)--cycle;
\draw[blue,thick](14.7,2)--(13.6,4);
\draw node at(11.7,1.5){\footnotesize $A'=A$};\draw node at(16.9,1.5){\footnotesize $B$};
\draw node at(14.3,6.2){\footnotesize $C$};
\draw node at(14.6,1.5){\footnotesize $B'$};\draw node at(13.1,4.1){\footnotesize $C'$};
\draw[<-](14.3,3)--(15.6,3.8);\draw node at(14.8,3.8){\footnotesize $x$};
\draw node at(14.7,0){\footnotesize $t>x+t>0$};

\draw[green](21.6,4)--(24.9,4)--(23.25,1);
\draw(29.3,4)--(27.1,8)--(24.9,4)--cycle;
\draw[blue,dashed, thick](21.6,4)--(23.25,1);
\draw node at(25,3.5){\footnotesize $A$};\draw node at(24.7,4.5){\footnotesize $A'$};
\draw node at(29.2,3.5){\footnotesize $B$};\draw node at(21.1,4){\footnotesize $B'$};
\draw node at(26.7,8.2){\footnotesize $C$};\draw node at(23.1,0.5){\footnotesize $C'$};
\draw[<-](22.6,2.6)--(28.1,5.9);\draw node at(23.1,3.3){\footnotesize $x$};
\draw node at(25.8,0){\footnotesize $x<-t<0$};
 \end{tikzpicture}
\caption{The triangle $\triangle A'B'C' = \< x,0,0,t>$ for different $x$.} \label{f5}
\end{center}
\end{figure}

\begin{figure}[htbp]
\begin{center} 
\begin{tikzpicture}[>=stealth,scale=0.40]

\draw(3.3,8)--(0,2)--(6.6,2);
\draw[blue,thick](3.3,8)--(6.6,2);
\draw(2.2,6)--(4.4,2);
\draw[yellow,thick](2.2,2)--(4.4,6);
\draw node at(-0.2,1.5){\footnotesize $A$};\draw node at(4.3,1.5){\footnotesize $B$};
\draw node at(1.8,6.1){\footnotesize $C$};
\draw node at(6.6,1.5){\footnotesize $B'$};\draw node at(5,6.2){\footnotesize $C'$};
\draw node at(2,1.5){\footnotesize $A'$};
\draw[->](3.8,3.5)--(5.2,4.2);\draw node at(4.2,4.1){\footnotesize $x$};
\draw[->](1.4,4.1)--(2.8,3.4);\draw node at(2.2,4.2){\footnotesize $y$};
\draw node at(3.3,0){\footnotesize $x>0,\;t>y+t>0$};

\draw[green](18.6,7)--(12,7)--(15.3,1);
\draw[](16.4,7)--(15.3,7);
\draw[blue,thick](15.85,8)--(17.5,5);
\draw[yellow,thick,dashed](18.6,7)--(15.3,1);
\draw(13.65,4)--(15.85,8);
\draw node at(13.2,3.7){\footnotesize $C$};\draw node at(11.6,7.3){\footnotesize $B$};
\draw node at(15,7.4){\footnotesize $A$};\draw node at(19.1,7.4){\footnotesize $A'$};
\draw node at(16.9,7.4){\footnotesize $B'$};\draw node at(18.1,4.8){\footnotesize $C'$};
\draw[->](12.9,5.6)--(16,7.45);\draw node at(13.5,6.4){\footnotesize $x$};
\draw[->](14.7,5.5)--(16.9,4.3);\draw node at(15.7,5.5){\footnotesize $y$};
\draw node at(14.6,0){\footnotesize $x>-t>0,\;y<0$};
 \end{tikzpicture}
\caption{The triangle $\triangle A'B'C' = \< x,y,0,t>$ for different $x,\; y$.} \label{f6}
\end{center}
\end{figure}

\begin{figure}[H]
\begin{center} 
\begin{tikzpicture}[>=stealth,scale=0.50]

\draw[green](.55,9)--(3.3,4);
\draw(3.85,9)--(0,2);
\draw(2.2,6)--(4.4,2);
\draw[brown,thick,dashed](.55,9)--(6.05,9);
\draw[blue,thick](2.85,9)--(4.4,6);\draw[blue,thick](4.4,6)--(6.6,2);
\draw[yellow,thick,dashed](4.4,6)--(6.05,9);\draw[yellow,thick](2.2,2)--(4.4,6);
\draw(0,2)--(6.6,2);
\draw node at(-.1,1.65){\footnotesize $A$};\draw node at(4.4,1.6){\footnotesize $B$};
\draw node at(1.75,6){\footnotesize $C$};
\draw node at(0.2,9.3){\footnotesize $B_z$};\draw node at(2.7,9.4){\footnotesize $B'$};
\draw node at(6.7,1.6){\footnotesize $B_x$};\draw node at(4.9,6.05){\footnotesize $C'$};
\draw node at(2,1.6){\footnotesize $A_y$};\draw node at(3.9,9.35){\footnotesize $A_z$};
\draw node at(6.2,9.4){\footnotesize $A'$};
\draw node at(3.75,3.95){\footnotesize $C_y$};\draw node at(3.8,8){\footnotesize $C_x$}; 

\end{tikzpicture}
\caption{The triangle $\triangle A'B'C' = \<x,y,z,t>$ \; for $x>0,\;t>y+t>0,\;$ \mbox{$z<-t<0$ }
where \mbox{$\triangle AB_xC_x=\< x,0,0,t>$},\; $\triangle BC_yA_y=\< 0,y,0,t>, \;
 \triangle CA_zB_z=\<0,0,z,t>$,\;
 $\triangle A_yB_xC'=\<x,y,0,t>$,\; $\triangle A'B_zC_y=\<0,y,z,t>$,\; $\triangle A_zB'C_x=\< x,0,z,t>$.} \label{f7}
\end{center}
\end{figure}
\noindent Let us consider Eq. (\ref{7}) for  following  numbers.
\begin{align}
  \<-1> = \<2,-2,-4,3 > =&\ \< 2,-2,0,3> + \<2,0,-4,3> +\<0,-2,-4,3> \nonumber \\
  						&\ - \<2,0,0,3> -\<0,-2,0,3>-\<0,0,-4,3> \nonumber\\
  						&\ +\<0,0,0,3>.  \label{8}
\end{align}

\noindent In Fig. ~\ref{f8} we can see individual components of Eq. (\ref{8}).
 The first (black) components are the sum $\< 2,-2,0,3> + \<2,0,-4,3> +\<0,-2,-4,3> +\<0,0,0,3>$,\\
 the second (red) components are the sum $- \<2,0,0,3> -\<0,-2,0,3>-\<0,0,-4,3>$ and the right side of the equation is  the  result $ \<2,-2,-4,3 >=\<-1>$.\\
  Numbers 2 staying on the triangles means an overlapping two  triangles. \\
 
 \begin{figure}[H]
\begin{center} 
\begin{tikzpicture}[>=stealth,scale=0.55]
\draw[fill,gray!60](3.3,4)--(4.4,4)--(3.85,3)--cycle;
\draw[green](3.3,4)--(4.4,4)--(3.85,3);
\draw[fill,gray!60](2.2,4)--(1.1,4)--(1.65,3)--cycle;
\draw[fill,gray!60](0,0)--(5.5,0)--(2.75,5)--cycle;
\draw[](0,0)--(5.5,0)--(2.75,5)--cycle; 
\draw[green](2.2,4)--(1.1,4)--(1.65,3);
\draw(2.2,0)--(3.85,3)--(3.3,4)--(2.2,4)--(1.65,3)--(3.3,0);
\draw[->](.3,1)--(1.5,1);
\draw node at(-1,1){\footnotesize $\<0,0,0,3>$};
\draw[->](1.1,3)--(2.5,3);
\draw node at(-.7,3){\footnotesize $\<0,-2,-4,3>$};
\draw[->](1,5)--(2.8,4.3);
\draw node at(-.6,5){\footnotesize $\<2,0,-4,3>$};
\draw[->](4.8,-.3)--(3.9,1);
\draw node at(6.3,-.5){\footnotesize $\<2,-2,0,3>$};
\draw node at (2.75,.4){\footnotesize 2}; 

\draw node at(6.5,2.5){$+$};
\draw[fill,red!30](10.2,4)--(9.1,4)--(9.65,3)--cycle;
\draw[green](10.2,4)--(9.1,4)--(9.65,3);
\draw[fill,red!30](8,0)--(13.5,0)--(10.75,5)--cycle; 
\draw[red](8,0)--(13.5,0)--(10.75,5)--cycle; 
\draw[red](10.2,0)--(10.75,1)--(11.3,0);
\draw node at(10.7,2){\footnotesize $-\<2,0,0,3>$};
\draw[->](8.7,4.5)--(9.8,3.7);
\draw node at(6.9,4.6){\footnotesize $-\<0,0,-4,3>$};
\draw[->](11.4,-.3)--(10.6,.2);
\draw node at(12.8,-.5){\footnotesize $-\<0,-2,0,3>$};
\draw node at (10.75,.5){\footnotesize 2};

\draw node at(14.5,2.5){$=$}; 
\draw[green](16,0)--(21.5,0)--(18.75,5)--cycle; 
\draw[green](18.2,4)--(17.1,4)--(17.65,3);
\draw[fill,gray!60](19.3,4)--(20.4,4)--(19.85,3)--cycle;
\draw[green](19.3,4)--(20.4,4)--(19.85,3);
\draw[->](19.6,4.9)--(19.9,3.7);
\draw node at(17.5,5.3){\footnotesize $\<2,-2,-4,3>=\<-1>$};
\end{tikzpicture}
\caption{} \label{f8}
\end{center}
\end{figure}

\noindent A simplified diagram of Eq. (\ref{f8}) can be seen in Fig. \ref{f9}.

\begin{figure}[H]
\begin{center} 
\begin{tikzpicture}[>=stealth,scale=0.55]
\draw[](0,0)--(5.5,0)--(2.75,5)--cycle; 
\draw[blue](5.5,0)--(2.75,5);
\draw(1.1,4)--(3.3,0);
\draw[yellow,thick](2.2,0)--(4.4,4);
\draw[brown, thick](1.1,4)--(4.4,4);
\draw node at(1.5,1){\footnotesize $\<3>$};
\draw[->](.9,3)--(2.5,3);
\draw node at(-.7,3){\footnotesize $\<0,-2,-4,3>$};
\draw[->](1,5)--(2.8,4.3);
\draw node at(-.5,5){\footnotesize $\<2,0,-4,3>$};
\draw[->](4.8,-.3)--(3.9,1);
\draw node at(6.3,-.5){\footnotesize $\<2,-2,0,3>$};
\draw[->](3.6,4.9)--(3.9,3.7);
\draw node at(3.6,5.3){\footnotesize $\<-1>$};
\end{tikzpicture}
\caption{} \label{f9}
\end{center}
\end{figure}

\noindent From the below equation
\[
\< -1>=\<-1,-1,-1,2>=3\<0>-3\< 1>+\<2>
\]
we  can get the way  of building of the triangle $\<-1>$.
Fig. \ref{f10} shows each step of  the construction of $\<-1>$.
\begin{figure}[H]
\begin{center} 
\begin{tikzpicture}[>=stealth,scale=0.8]

\draw[fill,gray!60](0,0.8)--(2,0.8)--(1,2.4)--cycle;
\draw(0,0.8)--(2,0.8)--(1,2.4)--cycle;\draw(1,0.8)--(1.5,1.6)--(0.5,1.6)--cycle;
\draw node at(1,0){$\< 2>$};
\draw[fill,gray!60](4.3,0.8)--(3.8,1.6)--(4.8,1.6)--cycle;
\draw[green](4.3,0.8)--(3.8,1.6)--(4.8,1.6)--cycle;
\draw[fill,red](4.3,.8)circle(1.5pt);
\draw[fill,red](3.8,1.6)circle(1.5pt);
\draw[fill,red](4.8,1.6)circle(1.5pt);
\draw node at(4.4,0){$\< 2>-3\<1>$};
\draw[fill,gray!60](8.5,1.6)--(9,0.8)--(9.5,1.6)--cycle;
\draw[green](8.5,1.6)--(9,0.8)--(9.5,1.6)--cycle;
\draw[fill,green](8.5,1.6)circle(1.5pt);
\draw[fill,green](9,.8)circle(1.5pt);
\draw[fill,green](9.5,1.6)circle(1.5pt);
\draw node at(9.2,0){$\< 2>-3\< 1>+3\<0>=\<-1> $};
\end{tikzpicture}
\caption{Stages of  the construction of the triangle $\<-1>$.}
\label{f10}
\end{center}
\end{figure}
\noindent Similarly,  from the relation
\[ \<-x>=\<-x,-x,-x,2x>= 3\< 0>-3\< x>+\<2x>\]
we can see that for each $x \in \mathbb{R},\; x>0$ the triangle  $ \< -x>$   is the opened triangle.\\

\section{The equation true in the geometric sense.}

\noindent Let us consider Eq. (\ref{7}) for  concrete numbers
\begin{align}
  \<4> = \<1,1,2,0 > =&\ \<1,1,0,0>+\<1,0,2,0>+\<0,1,2,0> \nonumber\\
  					  &\ -\<1,0,0,0>-\<0,1,0,0>-\<0,0,2,0>+\<0,0,0,0> \nonumber \\
  					 = &\ \<2>+2\< 3> - 2\<1> - \<2> +\<0>. \label{9}
\end{align}
After reduction we will get
\begin{align}
    \<4> = \<1,1,2,0 > &= 2\< 3>  -2\<1>+\<0>. \label{10}
\end{align}
\noindent From the arithmetic point of view Eq.  (\ref{10})
 is true. But it easy to see that we can  not build the triangle $\<4> $
 using only two  triangles $\<3> $,  two  triangles  $-\<1> $ and point $\<0>$. We need the triangles $\<2>  $ and $-\< 2 >$ too. They are not reducible to the empty set because they do not lie on top of each other   (Fig. \ref{f11}).\\
\begin{figure}[H]
\begin{center}
\begin{tikzpicture}[>=stealth,scale=0.7]

\draw[brown,thick](0,0)--(4.4,0);
\draw[blue,thick](4.4,0)--(2.2,4);
\draw[yellow,thick](2.2,4)--(0,0);
\draw(1.1,0)--(2.75,3);
\draw(3.3,0)--(1.65,3);
\draw(1.1,2)--(3.3,2);
\draw node at(1.65,2.3){\footnotesize$\<1>$};
\draw node at(2.2,0.5){\footnotesize$\<2>$};
\draw node at(2.75,2.3){\footnotesize$\<1>$};
\draw node at(1.2,1.1){\footnotesize$\<3>$};
\draw node at(3.2,1.1){\footnotesize$\<3>$};
\draw node at(2.2,3){\footnotesize$\<2>$};
\draw[->](0,1.5)--(2.1,1.95);
\draw node at(-.4,1.5){\footnotesize$\<0>$}; 

\draw[fill,gray!60](9.1,2)--(8,0)--(12.4,0)--(11.3,2)--cycle;
\draw(9.1,2)--(8,0)--(12.4,0)--(11.3,2);
\draw(9.1,0)--(10.2,2);
\draw(11.4,0)--(10.2,2);
\draw[green,thick](9.1,2)--(11.3,2);
\draw[fill,green](9.1,2)circle(1.5pt);
\draw[fill,green](11.3,2)circle(1.5pt);
\draw[fill](10.2,2)circle(1.5pt);
\draw node at(10.2,0.5){\footnotesize $2$};
\end{tikzpicture}
\caption{A geometric interpretation of Eq. (\ref{9}) and Eq. (\ref{10}).}
\label{f11}
\end{center}
\end{figure} 
\noindent But if we  put Eq. (\ref{9}) in the following equation
\begin{align}
    \<6> = \<2,2,2,0 > &= 3\< 4>  -3\<2>+\<0>, \label{102}
\end{align}
we get
\begin{align}
    \<6> = &\ 2\< 4> +2\<3> +\<2>-4\<2>-2\<1>+2\<0> \nonumber\\
    	= &\ 2\< 4> +2\<3> -3\<2>-2\<1>+2\<0>.	 \label{104}
\end{align}
We could reduce  $\<2>-\<2>$ in  Eq. (\ref{104}) because after reduction we can still build the  triangle $\<6>$ from the remaining elements (Fig. \ref{f11a}).\\
Let us note that $\<2>$ and $-\<2>$ lie on top of each other but $\<2>$ is from equation Eq. (\ref{9}) and $-\<2>$ is from equation Eq. (\ref{102})

\begin{figure}[H]
\begin{center}
\begin{tikzpicture}[>=stealth,scale=1.2]

\draw[fill,gray!60](1,0)--(0,0)--(1,1.72)--(1.5,.86)--cycle;
\draw(1,0)--(0,0)--(1,1.72)--(1.5,.86);
\draw[green,thick](1,0)--(1.5,.86);
\draw[fill,green](1.5,.86)circle(1.5pt);
\draw[fill,green](1,0)circle(1.5pt);

\draw[fill,gray!60](1.25,2.15)--(1.75,3.01)--(2.75,1.29)--(1.75,1.29)--cycle;
\draw(1.25,2.15)--(1.75,3.01)--(2.75,1.29)--(1.75,1.29);
\draw[green,thick](1.25,2.15)--(1.75,1.29);
\draw[fill,green](1.25,2.15)circle(1.5pt);
\draw[fill,green](1.75,1.29)circle(1.5pt);
\draw[dashed,gray](1.75,1.29)--(2.25,2.15);

\draw[fill,gray!60](3,.86)--(3.5,0)--(1.5,0)--(2,.86)--cycle;
\draw(3,.86)--(3.5,0)--(1.5,0)--(2,.86);
\draw[green,thick](3,.86)--(2,.86);
\draw[dashed,gray](2,0)--(2.5,.86)--(3,0);
\draw[fill](2.5,1.05)circle(1.5pt);
\draw[fill](1.75,1)circle(1.5pt);
\draw[fill,red](2.5,.86)circle(1.5pt);
\draw[fill,green](3,.86)circle(1.5pt);
\draw[fill,green](2,.86)circle(1.5pt);

\draw[->](.9,2.2)--(1.65,1.1);
\draw node at(.7,2.4){\footnotesize$\<0>$}; 
\draw[->](3.6,1.05)--(2.6,1.05);
\draw node at(3.9,1.05){\footnotesize$\<0>$}; 
\draw[->](3.9,.4)--(2.5,.4);
\draw node at(5,.4){\footnotesize$2\<3>-\<2>-2\<1>$}; 
\draw[->](3.3,1.9)--(2.1,1.9);
\draw node at(4,1.9){\footnotesize$\<4> -\<2>$}; 
\draw[->,gray](3.6,1.5)--(2.2,1.5);
\draw node at(4.2,1.5){\scriptsize$\<2> -\<2>$}; 
\draw[->](0,.6)--(.9,.6);
\draw node at(-.65,.6){\footnotesize$\<4> -\<2>$};
\end{tikzpicture}
\caption{The elements of Eq. (\ref{104}) which build the triangle $\<6>$.}
\label{f11a}
\end{center}
\end{figure}

\noindent So it makes sense to introduce the following definitions.
\begin{definition}
The equation $\< x> = \sum_j\alpha_j\<x_j>$, where $\alpha_j \in \{-1,1\}$ is true in the geometric sense if we can build the triangle $\< x>$ from the elements $\alpha_j\< x_j>$. \label{d2.1}
\end{definition}

\begin{definition}
The equation $\< x> = \sum_j\alpha_j\<x_j>$, where $\alpha_j \in \{-1,1\}$ is true in the arithmetic sense if the equations  $ x^i  = \sum_j\alpha_j x^i_j$, where $\alpha_j \in \{-1,1\}$,  $i=0,1,2$ \label{d2.2}
hold.
\end{definition}

\begin{corol}
If the equation is true in the geometric sense, it is also true in the arithmetic sense.
\end{corol}

\noindent You can see that  Eq. (\ref {9}) is true in the  geometric sense, while Eq. (\ref {10}) only in the arithmetic sense.\\

\noindent  Fig.~\ref{f12}  shows that for $x,y,z,t>0$  Eq. (\ref{7}) is true in the geometric sense.
 
\begin{figure}[H]
   \begin{center}
\begin{tikzpicture}[>=stealth,scale=0.4]   
\draw(0.55,1)--(7.15,1);\draw(2.2,0)--(4.95,5);\draw(4.4,0)--(2.2,4);
\draw[brown,thick](0,0)--(7.7,0);
\draw[blue,thick](7.7,0)--(3.85,7);
\draw[yellow,thick](3.85,7)--(0,0);

\draw node at(-0.2,0.5){\footnotesize $z$};
\draw node at(3.2,-0.6){\footnotesize $z+t$}; 
\draw node at(7.9,0.5){\footnotesize $z$};
\draw node at(1,-0.6){\footnotesize $y$}; 
\draw node at(0.2,2.4){\footnotesize $y+t$};
\draw node at(4.8,6.2){\footnotesize $y$};
\draw node at(6,-0.6){\footnotesize $x$};
\draw node at(2.5,5.5){\footnotesize $x$};
\draw node at(7.4,2.7){\footnotesize $x+t$};
\draw node at(4,1.5){\footnotesize $t$ };
\end{tikzpicture}
\caption{Interpretation of Eq. (\ref{7}) for $x,y,z,t>0$.}
\label{f12}
\end{center}
\end{figure}
\noindent We want to prove that Eq. (\ref{7}) is true in the geometric sense $\forall x,y,z,t \in \mathbb{R}$.

\noindent Let's transform  Eq. (\ref{7}).  
\begin{align}\begin{split}
    \<x,y,0,t>&= \< x,y,z,t> + \<x,0,0,t> + \< 0,y,0,t>   \\
             &\quad - \<x,0,z,t> -\< 0,y,z,t> - \< 0,0,0,t > +  \< 0,0,z,t >. 
                 \label{11}
 \end{split}   \end{align}
 Eq. (\ref{11}) does not describe the transition from  the triangle $\<t>$  to the triangle \mbox{$\<x+y+t>$} becouse it  contains terms with the variable $z$.
 Eq. (\ref{11})  is equivalent to the following  Eq. (\ref{12}) and describes the transition from  the triangle $\<z+t>$  to the  triangle $\<x+y+t>$.
\begin{align}\begin{split}
    \<x,y,-z,z+t>&= \< x,y,0,z+t> + \<x,0,-z,z+t> + \< 0,y,-z,z+t>   \\
             &\quad - \<x,0,0,z+t> -\< 0,y,0,z+t> - \< 0,0,-z,z+t >\\
             & \quad +\< 0,0,0,z+t > . 
                 \label{12}
\end{split}    \end{align} 
Just in  Eq. (\ref{7}), the triangle $ \<x,y,0,t>$  becomes the triangle $ \<x,y,-z,z+t>$. 
  
\noindent Similarly Eq. (\ref{13})  describes the transition from  the triangle $\<y+z+t>$  to the  triangle $\<x+t>$ and Eq. (\ref{14})  describes the transition from  the triangle $\<x+y+z+t>$ to the triangle $\<t>$.
\begin{align}\begin{split}
\<x+t>=& \<x,-y,-z,y+z+t>\\
      =& \< x,-y,0,y+z+t> + \<x,0,-z,y+z+t> + \< 0,-y,-z,y+z+t>   \\
             & - \<x,0,0,y+z+t> -\< 0,-y,0,y+z+t> - \< 0,0,-z,y+z+t >\\
             &  +\< 0,0,0,y+z+t > . 
                 \label{13}
\end{split}    \end{align} 
\begin{align}\begin{split}
\<t>=& \<-x,-y,-z,x+y+z+t>\\
      =& \< -x,-y,0,x+y+z+t> + \<-x,0,-z,x+y+z+t>\\
       & + \< 0,-y,-z,x+y+z+t> - \<-x,0,0,x+y+z+t>\\
       & -\< 0,-y,0,x+y+z+t> - \< 0,0,-z,x+y+z+t >\\
        & +\< 0,0,0,y+z+t > . 
                 \label{14}
\end{split}    \end{align} 

\noindent Below we will give two proofs of the fundamental theorem on the addition of triangles.\\

 \begin{thm}
      $\forall x,y,z,t \in \mathbb{R}$ Equation (\ref{7}) is true in the geometric sense. \label{t2.4}
\end{thm}
\begin{proof}
\noindent If in Eq. (\ref{7}) one element, for example,  $z$ is negative we  can replace \mbox{Eq.  (\ref{7})} with \mbox{Eq. (\ref{12})}.\\
If in Eq. (\ref{7}) two elements  $y,z$ are negative we  can replace \mbox{Eq.  (\ref{7})} with \mbox{Eq. (\ref{13})}.\\
If  Eq. (\ref{7}) has three elements  $x,y,z$ negative we  will take Eq. (\ref{14}).\\
So it is sufficient to consider the  cases $x>0,\, y>0,\, z>0$ and any $t$.
We have 10 following  cases.
\begin{enumerate}
\item[(1)] $t >0$, (Fig.~\ref{f12})\\
\text{In next cases  $t<0$}.
\item[(2)] $x+t>0,\, y+t>0,\, z+t>0$, (Fig.~\ref{f13}).
\item[(3)] $x+t>0,\, y+t>0,\, z+t<0$, (Fig.~\ref{f14}).
\item[(4)] $x+t>0,\, y+t<0,\, z+t<0,\, y+z+t>0$, (Fig.~\ref{f15}).
\item[(5)] $x+t>0,\, y+z+t<0$, (Fig.~\ref{f16}).
\item[(6)] $x+t<0,\, y+t<0,\, z+t<0,\, x+y+t>0,\, y+z+t>0,\, x+z+t>0$, (Fig.~\ref{f17}).
\item[(7)] $x+t<0,\, x+y+t>0,\, y+z+t<0,\, x+z+t>0$, (Fig.~\ref{f18}).
\item[(8)] $ x+y+t>0,\, y+z+t<0,\, x+z+t<0$, (Fig.~\ref{f19}).
\item[(9)] $x+y+t<0,\,y+z+t<0,\,x+z+t<0,\,x+y+z+t>0$, (Fig.~\ref{f20}).
\item[(10)] $x+y+z+t<0$, (Fig.~\ref{f21}).
\end{enumerate}
The proof is based on reviewing each figure and founding that by using the components of Eq. (~\ref {7}), we always get the  triangle
$\<x+y+z+t>$ from the triangle $\<t>$.
In Figs. \ref{f13}-\ref{f21} the  triangle $\<x+y+z+t>$ is denoted by  $\<t' >$.
\begin{figure}[htbp]
\begin{minipage}[t]{.45\textwidth}
 \begin{center}
 \begin{tikzpicture}[>=stealth,scale=0.4] 
 \draw[brown,thick](0,0)--(7.7,0);
\draw[blue,thick](7.7,0)--(3.85,7);
\draw[yellow,thick](3.85,7)--(0,0);  
 
\draw(3.3,0)--(5.5,4);
\draw(2.2,4)--(4.4,0);
\draw(1.65,3)--(6.05,3);
\draw [decorate,decoration={brace,amplitude=5pt},
       xshift=4pt,yshift=2pt]
  (4,7.1)--(7.9,0.1); 
\draw node at(7.2,4){\footnotesize $\<t'>$};
\draw node at(3.85,2.25){\footnotesize $\<t>$};
\end{tikzpicture}
\caption{$x+t>0,\, y+t>0,\, z+t>0.$}
\label{f13}
\end{center}
\end{minipage}
\hfill
\begin{minipage}[t]{.45\textwidth}
 \begin{center}
 \begin{tikzpicture}[>=stealth,scale=0.4]
 \draw[brown,thick](0,1)--(7.7,1);
\draw[blue,thick](7.7,1)--(3.85,8);
\draw[yellow,thick](3.85,8)--(0,1);  

\draw(1.1,3)--(6.6,3);
\draw(1.65,4)--(3.85,0)--(6.05,4);
\draw [decorate,decoration={brace,amplitude=5pt},
       xshift=4pt,yshift=2pt]
  (4,8.1)--(7.9,1.1); 
 \draw node at(7.2,5){\footnotesize $\<t'>$};
\draw node at(3.9,2){\footnotesize $\<t>$};
 
\end{tikzpicture}
\caption{$x+t>0,\, y+t>0,\,z+t<0.$}
\label{f14}
\end{center}
\end{minipage}
\end{figure}

\begin{figure}[htbp]
\begin{minipage}[t]{.46\textwidth}
 \begin{center}
\begin{tikzpicture}[>=stealth,scale=0.5]
 \draw[brown,thick](0,1)--(5.5,1);
\draw[blue,thick](5.5,1)--(2.75,6);
\draw[yellow,thick](2.75,6)--(0,1); 
\draw(3.85,4)--(1.65,0)--(0,3)--(4.4,3);
\draw [decorate,decoration={brace,amplitude=5pt},
       xshift=4pt,yshift=2pt]
  (2.9,6.1)--(5.6,1.1); 
 \draw node at(5.3,4){\footnotesize $\<t'>$}; 
 \draw node at(1.7,1.9){\footnotesize $\<t>$};
\end{tikzpicture}
\caption{$x+t>0,\; y+t<0,$\; \mbox{$ z+t<0,$}$\;y+z+t>0.$}
\label{f15}
\end{center}
\end{minipage}
\hfill
\begin{minipage}[t]{.45\textwidth}
 \begin{center}
 \begin{tikzpicture}[>=stealth,scale=0.6]
  \draw[brown,thick](1.55,2)--(5.95,2);
\draw[blue,thick](5.95,2)--(4.3,5);
\draw[yellow,thick](2.1,1)--(4.3,5); 

\draw(5.4,3)--(1,3)--(2.65,0)--(4.85,4);
\draw [decorate,decoration={brace,amplitude=5pt},
       xshift=4pt,yshift=2pt]
  (4.4,5.1)--(6.1,2.1);
  \draw node at(6.1,4){\footnotesize $\<t'>$}; 
\draw [decorate,decoration={brace,amplitude=5pt},
       xshift=4pt,yshift=2pt]
 (2.35,-0.2)--(0.65,2.9);
\draw node at(1,1.1){\footnotesize $\<t>$};
\end{tikzpicture}
\caption{$x+t>0,\; y+z+t<0.$}
\label{f16}
\end{center}
\end{minipage}
\end{figure}

\begin{figure}[H]
\begin{minipage}[t]{.45\textwidth}
\begin{center}
 \begin{tikzpicture}[>=stealth,scale=0.3]
 \draw[brown,thick](3,3)--(9.6,3);
\draw[blue,thick](9.6,3)--(6.3,9);
\draw[yellow,thick](3,3)--(6.3,9); 
\draw(3,7)--(9.6,7)--(6.3,1)--(3,7)--cycle;
\draw [decorate,decoration={brace,amplitude=5pt},
       xshift=4pt,yshift=2pt]
  (8.3,9.9)--(11.6,4);
  \draw node at(11.5,7.8){\footnotesize $\<t'>$}; 
\draw [decorate,decoration={brace,amplitude=5pt},
       xshift=4pt,yshift=2pt]
 (4.2,-.1)--(.9,5.8);
\draw node at(1.3,2.5){\footnotesize $\<t>$};
\end{tikzpicture}
\caption{    
$x+t<0,$ \mbox{$ y+t<0,$} \mbox{$z+t<0,$}\mbox{$ x+y+t>0$}, \mbox{$y+z+t>0,$} \mbox{$x+z+t>0.$}}
\label{f17}
\end{center}
\end{minipage}
\hfill
\begin{minipage}[t]{.46\textwidth}
 \begin{center}
 \begin{tikzpicture}[>=stealth,scale=0.4]
 \draw[brown,thick](2.1,3)--(6.5,3);
\draw[blue,thick](6.5,3)--(4.85,6);
\draw[yellow,thick](4.85,6)--(2.65,2); 
 \draw(1,5)--(6.5,5)--(3.75,0)--cycle;
\draw [decorate,decoration={brace,amplitude=5pt},
       xshift=4pt,yshift=2pt]
 (3.3,-0.1)--(0.55,4.9);
\draw node at(.9,2){\footnotesize $\<t>$};

\draw node at(4.8,4){$\<t'>$};
\end{tikzpicture}
\caption{$x+t<0,\, x+y+t>0,$ \mbox{ $y+z+t<0$},\;$x+z+t>0$.}
\label{f18}
\end{center}
\end{minipage}
\end{figure}

\begin{figure}[H]
\begin{minipage}[t]{.45\textwidth}
\begin{center}
 \begin{tikzpicture}[>=stealth,scale=0.5]
\draw(1,5)--(6.5,5)--(3.75,0)--cycle;
 \draw[brown,thick](1.55,4)--(5.95,4);
\draw[blue,thick](3.85,6)--(5.4,3);
\draw[yellow,thick](2.1,3)--(3.85,6); 
\draw [decorate,decoration={brace,amplitude=5pt},
       xshift=4pt,yshift=2pt]
 (3.4,-0.2)--(0.65,4.9);
\draw node at(1.3,2){\footnotesize $\<t>$};
\draw node at(3.8,4.5){\footnotesize $\<t'>$};

\end{tikzpicture}
\caption{$ x+y+t>0,$ \mbox{$y+z+t<0,$} $x+z+t<0.$}
\label{f19}
\end{center}
\end{minipage}
\hfill
\begin{minipage}[t]{.45\textwidth}
 \begin{center}
 \begin{tikzpicture}[>=stealth,scale=0.4]
 \draw[brown,thick](2.65,4)--(7.15,4);
\draw[blue,thick](4.3,7)--(6.5,3);
\draw[yellow,thick](3.2,3)--(5.4,7); 
\draw(1,7)--(8.7,7)--(4.85,0)--cycle;
\draw [decorate,decoration={brace,amplitude=5pt},
       xshift=4pt,yshift=2pt]
 (4.5,-0.2)--(0.65,6.9);
\draw node at(1.6,2.9){\footnotesize $\<t>$};
\draw node at(4.85,4.65){\footnotesize $\<t'>$};
\end{tikzpicture}
\caption{$x+y+t<0,$ \mbox{$y+z+t<0,$} $x+z+t<0, \;x+y+z+t>0.$}
\label{f20} 
\end{center}
\end{minipage}
\end{figure}

\begin{figure}[H]
\begin{center}
\begin{tikzpicture}[>=stealth,scale=0.5]
\draw(1,5)--(6.5,5)--(3.75,0)--cycle;
 \draw[brown,thick](1.55,4)--(5.95,4);
\draw[blue,thick](2.1,5)--(4.3,1);
\draw[yellow,thick](3.2,1)--(5.4,5); 
\draw [decorate,decoration={brace,amplitude=5pt},
       xshift=4pt,yshift=2pt]
 (3.4,-0.2)--(0.65,4.9);
\draw node at(1.3,2){\footnotesize $\<t>$};
\draw node at(3.75,3.4){\footnotesize $\<t'>$};
\end{tikzpicture}
\caption{$x+y+z+t<0.$}
\label{f21}
\end{center}
\end{figure}

\noindent One can show that the case (8)  can be replaced by the case (3). \\
Indeed, if we replace Eq. (\ref{7}) by  equivalent Eq. (\ref{14}) then
$\<t>$ acts as \\ $\< t'> = \<x+y+z+t>$ and $\<z+t> = \<-x,-y,0,(x+y+z+t)>$ acts as $\<x+y+t>$.
In the case (3) only $t'$ and $x+y+t$ are positive and in the case (8) only $t$ and $l+t$ are negative. By changing the sign in all components of the case (8) we will receive the case (3).\\

\noindent  Similarly we can ignore cases (7), (9) and (10), which are equivalent to cases (4), (2) and (1) respectively. And so we have 6 different cases represented by 6 different figures.

\noindent In Figs. ~\ref{f22}-~\ref{f27} we can see individual terms of Eq. (\ref{7}) for cases (1)-(6).
 The  first (black) component is the sum $\<x+y+t> + \<x+z+t> +\<y+z+t>$,
 the second (red) component is the sum $-\<x+t> - \<y+t> -\<z+t>$, the third component is  $\<t> $ and the right side of the equation is  the  triangle $\<x+y+z+t>$. Numbers 2 or 3 staying on the triangles mean an overlapping two or three triangles.

\begin{figure}[H]
\begin{center}
\begin{tikzpicture}[>=stealth,scale=0.7]
\draw[fill,gray!60](0,0)--(4,0)--(2,3.2)--cycle;
\draw[thick](0,0)--(4,0)--(2,3.2)--cycle;
\draw(1,0)--(2.5,2.4)--(3.5,0.8)--(0.5,0.8)--(1.5,2.4)--(3,0);
\draw node at(2,1.1){3};\draw node at(1.3,1.3){2};
\draw node at(2.7,1.3){2};\draw node at(2,.4){2};

\draw node at(4,1.1){$+$};
\draw[fill,red!30](5,0)--(6.5,2.4)--(7.5,0.8)--(4.5,0.8)--(5.5,2.4)--(7,0)--cycle;
\draw[red,thick](5,0)--(6.5,2.4)--(7.5,0.8)--(4.5,0.8)--(5.5,2.4)--(7,0)--cycle;
\draw node at(6,1.1){3};

\draw node at(8,1.1){$+$};
\draw[fill,gray!60](8.5,0.8)--(9.5,0.8)--(9,1.6)--cycle;
\draw[thick](8.5,0.8)--(9.5,0.8)--(9,1.6)--cycle;

\draw node at(10,1.1){$=$};
\draw[fill,gray!60](10,0)--(14,0)--(12,3.2)--cycle;
\draw[thick](10,0)--(14,0)--(12,3.2)--cycle;
\end{tikzpicture}
\caption{Equation (\ref{7}) for $x>0,\; y>0,\; z>0,\; t >0$.}
\label{f22}
\end{center}
\end{figure}

\begin{figure}[H]
\begin{center}
\begin{tikzpicture}[>=stealth,scale=0.7]
\draw[fill,gray!60](0,0)--(3,0)--(1.5,2.4)--cycle;
\draw[fill,gray!60](2,0)--(5,0)--(3.5,2.4)--cycle;
\draw[fill,gray!60](1,1.6)--(4,1.6)--(2.5,4)--cycle;

\draw[thick](0,0)--(5,0)--(2.5,4)--cycle;
\draw[thick](2,0)--(3.5,2.4)--(4,1.6)--(1,1.6)--(1.5,2.4)--(3,0)--cycle;
\draw node at(1.5,1.9){2};
\draw node at(3.5,1.9){2};
\draw node at(2.5,0.3){2};

\draw node at(5.1,1.1){$+$};
\draw[fill,red!30](5.5,1.6)--(6.5,1.6)--(6,2.4)--cycle;
\draw[red,thick](5.5,1.6)--(6.5,1.6)--(6,2.4)--cycle;
\draw[fill,red!30](7.5,1.6)--(8.5,1.6)--(8,2.4)--cycle;
\draw[red,thick](7.5,1.6)--(8.5,1.6)--(8,2.4)--cycle;
\draw[fill,red!30](6.5,0)--(7.5,0)--(7,0.8)--cycle;
\draw[red,thick](6.5,0)--(7.5,0)--(7,0.8)--cycle;

\draw node at(9.1,1.1){$+$};

\draw[fill,gray!60](9.5,1.6)--(10.5,1.6)--(10,0.8)--cycle;
\draw[thick,green](9.5,1.6)--(10.5,1.6)--(10,0.8)--cycle;

\draw node at(10.8,1.1){$=$};
\draw[fill,gray!60](11,0)--(16,0)--(13.5,4)--cycle;
\draw[thick](11,0)--(16,0)--(13.5,4)--cycle;
\end{tikzpicture}
\caption{Equation (\ref{7}) for $x+t>0,\; y+t>0,\; z+t>0,\; t<0$.}
\label{f23}
\end{center}
\end{figure}

\begin{figure}[H]
\begin{center}
\begin{tikzpicture}[>=stealth,scale=0.7]
\draw[fill,gray!60](0,0.8)--(2,0.8)--(1,2.4)--cycle;
\draw[fill,gray!60](3,0.8)--(5,0.8)--(4,2.4)--cycle;
\draw[fill,gray!60](0.5,1.6)--(4.5,1.6)--(2.5,4.8)--cycle;
\draw[thick](0,0.8)--(2,0.8)--(1,2.4)--cycle;
\draw[thick](3,0.8)--(5,0.8)--(4,2.4)--cycle;
\draw[thick](0.5,1.6)--(4.5,1.6)--(2.5,4.8)--cycle;
\draw node at(1,1.9){2};
\draw node at(4,1.9){2};

\draw node at(5.3,1.1){$+$};
\draw[fill,red!30](5.5,1.6)--(6.5,1.6)--(6,2.4)--cycle;
\draw[red,thick](5.5,1.6)--(6.5,1.6)--(6,2.4)--cycle;
\draw[fill,red!30](8.5,1.6)--(9.5,1.6)--(9,2.4)--cycle;
\draw[red,thick](8.5,1.6)--(9.5,1.6)--(9,2.4)--cycle;
\draw[fill,red!30](7,0.8)--(8,0.8)--(7.5,0)--cycle;
\draw[green,thick](7,0.8)--(8,0.8)--(7.5,0)--cycle;

\draw node at(9.9,1.1){$+$};
\draw[fill,gray!60](10.2,1.6)--(12.2,1.6)--(11.2,0)--cycle;
\draw[thick,green](10.2,1.6)--(12.2,1.6)--(11.2,0)--cycle;

\draw node at(12.5,1.1){$=$};
\draw[fill,gray!60](13,0.8)--(18,0.8)--(15.5,4.8)--cycle;
\draw[thick](13,0.8)--(18,0.8)--(15.5,4.8)--cycle;
\end{tikzpicture}
\caption{Equation (\ref{7}) for $x+t>0,\; y+t>0,\; z+t<0,\; t<0$.}
\label{f24}
\end{center}
\end{figure}

\begin{figure}[H]
\begin{center}
\begin{tikzpicture}[>=stealth,scale=0.7]
\draw[fill,gray!60](0,0.8)--(1,0.8)--(0.5,1.6)--cycle;
\draw[fill,gray!60](2,0.8)--(5,0.8)--(2.5,4.8)--(1,2.4)--(3,2.4)--cycle;
\draw[thick](0,0.8)--(1,0.8)--(0.5,1.6)--cycle;
\draw[thick](2,0.8)--(5,0.8)--(2.5,4.8)--(1,2.4)--(4,2.4)--(3.5,3.2)--cycle;
\draw node at(3.5,2.7){2};

\draw node at(5.2,1.5){$+$};
\draw[fill,red!30](5.5,2.4)--(6.5,2.4)--(6,1.6)--cycle;
\draw[thick,green](5.5,2.4)--(6.5,2.4)--(6,1.6)--cycle;
\draw[fill,red!30](6.5,0.8)--(7.5,0.8)--(7,0)--cycle;
\draw[thick,green](6.5,0.8)--(7.5,0.8)--(7,0)--cycle;
\draw[fill,red!30](8.5,2.4)--(9.5,2.4)--(9,3.2)--cycle;
\draw[red,thick](8.5,2.4)--(9.5,2.4)--(9,3.2)--cycle;

\draw node at(9.9,1.5){$+$};
\draw[fill,gray!60](10.2,2.4)--(13.2,2.4)--(11.7,0)--cycle;
\draw[thick,green](10.2,2.4)--(13.2,2.4)--(11.7,0)--cycle;

\draw node at(13.3,1.5){$=$};
\draw[fill,gray!60](13.5,0.8)--(18.5,0.8)--(16,4.8)--cycle;
\draw[thick](13.5,0.8)--(18.5,0.8)--(16,4.8)--cycle;
\end{tikzpicture}
\caption{Equation (\ref{7}) for $x+t>0,\; y+t<0,\; z+t<0,\; t<0$.}
\label{f25}
\end{center}
\end{figure}

\begin{figure}[H]
\begin{center}
\begin{tikzpicture}[>=stealth,scale=0.7]
\draw[fill,gray!60](0,1.6)--(1,1.6)--(0.5,0.8)--cycle;
\draw[thick,green](0,1.6)--(1,1.6)--(0.5,0.8)--cycle;
\draw[fill,gray!60](2,1.6)--(4,1.6)--(2.5,4)--(1.5,2.4)--(2.5,2.4)--cycle;
\draw[thick](2,1.6)--(4,1.6)--(2.5,4)--(1.5,2.4)--(3.5,2.4)--(3,3.2)--cycle;
\draw node at(3,2.7){2};

\draw node at(4.55,1.8){$+$};
\draw[fill,red!30](5,2.4)--(7,2.4)--(6.5,1.6)--(7.5,1.6)--(6.5,0)--cycle;
\draw[red,thick](5.5,1.6)--(6.5,1.6); \draw[red,thick](6.5,1.6)--(6,.8);
\draw[thick,green](5,2.4)--(7,2.4)--(6.5,1.6)--(7.5,1.6)--(6.5,0)--cycle;
\draw[fill,red!30](8,2.4)--(9,2.4)--(8.5,3.2)--cycle;
\draw[red,thick](8,2.4)--(9,2.4)--(8.5,3.2)--cycle;
\draw node at(6,1.3){2};

\draw node at(9.4,1.8){$+$};
\draw[fill,gray!60](9.7,2.4)--(12.7,2.4)--(11.2,0)--cycle;
\draw[thick,green](9.7,2.4)--(12.7,2.4)--(11.2,0)--cycle;

\draw node at(13.1,1.8){$=$};
\draw[fill,gray!60](13.5,0.8)--(18.5,0.8)--(16,4.8)--cycle;
\draw[thick](13.5,0.8)--(18.5,0.8)--(16,4.8)--cycle;
\end{tikzpicture}
\caption{Equation (\ref{7}) for $x+t>0,\; y+z+t<0,\; t<0$.}
\label{f26}
\end{center}
\end{figure}

\begin{figure}[H]
\begin{center}
\begin{tikzpicture}[>=stealth,scale=0.7]
\draw[fill,gray!60](0,0.8)--(1,0.8)--(0.5,1.6)--cycle;
\draw[thick](0,0.8)--(1,0.8)--(0.5,1.6)--cycle;
\draw[fill,gray!60](2,0.8)--(3,0.8)--(2.5,1.6)--cycle;
\draw[thick](2,0.8)--(3,0.8)--(2.5,1.6)--cycle;
\draw[fill,gray!60](1,2.4)--(2,2.4)--(1.5,3.2)--cycle;
\draw[thick](1,2.4)--(2,2.4)--(1.5,3.2)--cycle;

\draw node at(3.5,1.5){$+$};
\draw[fill,red!30](4,2.4)--(5,2.4)--(4.5,1.6)--cycle;
\draw[thick,green](4,2.4)--(5,2.4)--(4.5,1.6)--cycle;
\draw[fill,red!30](6,2.4)--(7,2.4)--(6.5,1.6)--cycle;
\draw[thick,green](6,2.4)--(7,2.4)--(6.5,1.6)--cycle;
\draw[fill,red!30](5,0.8)--(6,0.8)--(5.5,0)--cycle;
\draw[thick,green](5,0.8)--(6,0.8)--(5.5,0)--cycle;

\draw node at(7.5,1.5){$+$};
\draw[fill,gray!60](8,2.4)--(11,2.4)--(9.5,0)--cycle;
\draw[thick,green](8,2.4)--(11,2.4)--(9.5,0)--cycle;

\draw node at(11.5,1.5){$=$};
\draw[fill,gray!60](12,0.8)--(15,0.8)--(13.5,3.2)--cycle;
\draw[thick](12,0.8)--(15,0.8)--(13.5,3.2)--cycle;
\end{tikzpicture}
\caption{Equation (\ref{7}) for $x+t<0,\, y+t<0,\, z+t<0,\, x+y+t>0,$  \mbox{$ y+z+t>0,$} $\;x+z+t>0,\; t<0$.}
\label{f27}
\end{center}
\end{figure}

\noindent Thus, in all 6 cases covering the full space of real inputs, the triangle $\<x+y+z+t>$ can be constructed from $\<t>$ using the given identity. Therefore, Eq. (\ref{7}) holds geometrically for all real $x,\ y,\ z,\ t$.
 \end{proof}
 
 \noindent So the result of the equation (\ref{7}) is always a triangle.\\
 
\noindent 
Now we will present a second, more general proof of Theorem ~\ref{t2.4}, but first we need to make some preparations.\\

\noindent Let us note that all 8 triangles $T_j$, $j=1,2,\ldots,8$ of equation ~(\ref{7}) arise from the intersection of three pairs of parallel lines. Indeed, the three lines $a_1, b_1, c_1$ form the triangle $<t>$, while the other three lines $a_2, b_2, c_2$, parallel respectively to $a_1, b_1, c_1$, form the triangle $<x,y,z,t>$. The sides of the remaining triangles of equation ~(\ref{7}) lie on other combinations of the lines $a_i, b_i, c_i$, where $i=1,2$.
If three non-parallel lines intersect at a single point, then some of the triangles of equation ~(\ref{7}) degenerate into points.\\

\noindent Let us assume that the triangles $\<x>$ are closed for $x>0$
 and open for $x<0$.\\
Each of the eight triangles is either indivisible or consists of smaller polygons that are themselves indivisible. We call these smallest indivisible polygons (including indivisible triangles) cells (atoms). If a polygon shares an edge with a closed triangle, we regard that edge as belonging to the triangle, not to the polygon. If a polygon shares an edge with an open triangle, we regard that edge as belonging to the polygon. If a polygon is a part of a closed triangle, we consider its remaining sides to be closed. If a polygon is a part of an open triangle, we consider its remaining sides to be open.

\noindent Let us count the number of the cells.\\ 
Two pairs of parallel lines $a_1, b_1, a_2, b_2$ form a parallelogram (Fig.~\ref{f121}).
Through two vertices of the parallelogram we draw the axis $l_1$,
 and through the two remaining vertices of the parallelogram we draw the axes $l_2$ and $l_3$, parallel to the axis $l_1$.
Then the three pairs of parallel lines 
 form three cells $C_1, C_2, C_3$ containing the axis $l_1$,  two cells $C_4, C_5$ containing the axis $l_2$ and two cells $C_6, C_7$ containing the axis  $l_3$ (Fig.~\ref{f121}).

 \begin{figure}[H]
   \begin{center}
\begin{tikzpicture}[>=stealth,scale=0.6]  
\draw(-.55,-1)--(4.4,8);\draw(8.25,-1)--(3.3,8); 
\draw(-1,0)--(8.7,0);
\draw(-1,3)--(8.7,3);\draw(2.15,-1)--(7.1,8);
\draw(4.95,-1)--(0,8);
\draw[red,dashed,thick](3.95,8)--(3.45,-1);
\draw[green,dashed,thick](2.3,6)--(1.9,-.5);\draw[green,dashed,thick](5.3,6)--(4.95,-.5);

\draw node at(3,8.3){\footnotesize $a_1$};
\draw node at(-.1,8.3){\footnotesize $a_2$}; 
\draw node at(4.9,8.3){\footnotesize $b_1$};
\draw node at(7.4,8.3){\footnotesize $b_2$}; 
\draw node at(9.3,0){\footnotesize $c_1$};
\draw node at(9.3,3){\footnotesize $c_2$};
\draw node at(4,8.5){\footnotesize $l_1$};
\draw node at(2.4,6.5){\footnotesize $l_2$};
\draw node at(5.4,6.5){\footnotesize $l_3$};
\draw node at(4.5,4.5){\footnotesize $C_1$};
\draw node at(3.6,2.5){\footnotesize $C_2$};
\draw node at(3.8,.5){\footnotesize $C_3$};
\draw node at(2.3,3.3){\footnotesize $C_4$};
\draw node at(1.3,1){\footnotesize $C_5$};
\draw node at(5.5,3.35){\footnotesize $C_6$};
\draw node at(6,1){\footnotesize $C_7$};
\end{tikzpicture}
\caption{7 cells between three pairs of lines.}
\label{f121}
\end{center}
\end{figure}

\begin{proof}[Proof of Theorem 2.4]
Each of the eight triangles of equation (\ref{7}) is a uniquely defined sum of a subset of the seven atoms.  
We construct the matrix
\[
M \in \{0,1\}^{7 \times 8},
\]
where $M_{i,j}=1$ if the cell $C_i$ contributes to the triangle $T_j$, and $0$ otherwise.  
The rank of $M$ is at most $7$. Hence, every triangle is a rational linear combination of the other triangles.  

\noindent Thus, there exist rational coefficients $A,a,b,c,d,e,f,g$ such that the following equation holds in the geometric sense:
\begin{align}\begin{split}
  A \<x,y,z,t>&= a\< x,y,0,t> + b\<x,0,z,t> + c\< 0,y,z,t>  \\
             &\quad - d\<x,0,0,t> - e\< 0,y,0,t> - f\< 0,0,z,t> 
              + g\< 0,0,0,t>. 
      \label{141}        
\end{split}\end{align}

\noindent From this it follows that
\begin{align}\begin{split}
  A (x+y+z+t)^2 &= a(x+y+t)^2 + b(x+z+t)^2 + c(y+z+t)^2  \\
             &\quad - d(x+t)^2 - e(y+t)^2 - f(z+t)^2 
              + g t^2.
      \label{142}       
\end{split}\end{align}

\noindent We also know that equation(\ref{7}) is true in  algebraic sense. Thus,
\begin{align}\begin{split}
   (x+y+z+t)^2 &= (x+y+t)^2 + (x+z+t)^2 + (y+z+t)^2  \\
             &\quad - (x+t)^2 - (y+t)^2 - (z+t)^2 
              + t^2.
      \label{143}        
\end{split}\end{align}

\noindent Let us compute $(\ref{141}) - a(\ref{142})$:
\begin{align}\begin{split}
(A-a)(x+y+z+t)^2 &= (b-a)(x+z+t)^2 + (c-a)(y+z+t)^2  \\
             &\quad - (d-a)(x+t)^2 - (e-a)(y+t)^2 \\
             &\quad - (f-a)(z+t)^2 + (g-a)t^2.
       \label{144}       
\end{split}\end{align}

\noindent On the left-hand side of (\ref{144}) we have the term $(A-a)2xy$, but it does not appear on the right-hand side. Therefore, $A=a$.  
\noindent By a similar argument, we conclude that $A=b=c$.  

\noindent We thus obtain
\begin{align*}\begin{split}
  A (x+y+z+t)^2 &= A(x+y+t)^2 + A(x+z+t)^2 + A(y+z+t)^2  \\
             &\quad - d(x+t)^2 - e(y+t)^2 - f(z+t)^2 
              + g t^2.
\end{split}\end{align*}

\noindent But then  
\[
2Axt = 4Axt - 2dxt,
\]
which implies $A=d$.  
Similarly,  $A=e=f$.  
Finally, we also obtain $A=g$.  

\noindent Therefore, equation~(\ref{141}), true in the geometric sense, takes the form
\begin{align*}\begin{split}
A \<x,y,z,t>&= A\< x,y,0,t> + A\<x,0,z,t> + A\< 0,y,z,t>  \\
             &\quad - A\<x,0,0,t> - A\< 0,y,0,t> - A\< 0,0,z,t> 
              + A\< 0,0,0,t>,
  \end{split}\end{align*}
for $A \neq 0$. This completes the proof.
\end{proof}

\begin{rem}
Let us note that only in the case (6) no triangles of the same sign do not overlap (Fig. ~\ref{f27}).
\end{rem} 
\begin{corol}
      $\forall x,y,z,t \in \mathbb{R}$\\
       If $x+y+z=0$, then the triangles     $\<t>$ and  $\<x,y,z,t>$ from Eq.  (\ref{7}) are congruent.
\end{corol}

\noindent We can see that for  $x+y+z=0$ the expression $ \<x,y,z, \_ >  $   can be treated as a translation vector.  

\begin{thm}
 For every two triangles $ABC$ and  $A'B'C'$ with respective parallel sides there exist numbers $x,y,z,t \in \mathbb{R}$ such that\\
 $\quad \bigtriangleup ABC=\< t >, \qquad \bigtriangleup A'B'C'=\< x,y,z,t >$.
\end{thm}
\begin{proof}
The proof follows from the geometric construction of Eq.  (\ref{7}).
 \end{proof}
\begin{remark}
To get the triangle $ \< x,y,z,t>$ from the triangle $ \<  t>$ we can successively  build triangles
 $ \<x,0,0, t>$,  $ \< 0,y,0,(x+t)>$, \mbox{$ \< 0,0,z,(x+y+t)>$}.
\end{remark} 
 
\noindent Let us note that the Eq. (\ref{7}) is also true in the geometric sense.\\
\noindent For $n= \frac{1}{2}$  Eq. (\ref{7}) will take the following form.
 \begin{equation*}
 \<\frac{1}{2}>=\frac{3}{8}\<1>+\frac{3}{4}\<0>-\frac{1}{8}\<-1>.
 \end{equation*}
We will not have to look for an interpretation of the number $\<\frac{1}{2}>$ if we find that
\begin{equation*}
 \<1>=\frac{3}{8}\<2>+\frac{3}{4}\<0>-\frac{1}{8}\<-2>.
 \end{equation*}
 Just the number $\<\frac{1}{2}>$  is a triangle with a side $\frac{1}{2}$.\\
Similarly for each real number $x \in \mathbb{R}$ the notation\\
\begin{align}
  \<x>=\frac{x(x+1)}{2}\<1>-(x^2-1)\<0>+\frac{x(x-1)}{2}\<-1>
      \label{15}
\end{align}
  is another form of the  triangle $\<x>=(x^2,x,1)$ with a side $x$. 
 
\section{Multiplication in the set $\mathbb{R}_2$.}

\noindent Let us fix a point in $\mathbb{R}_2$ and denote it as $\<0,0,0>$. Then, using Eq. (\ref{7}), every element of the set $\mathbb{R}_2$ can be obtained from the distinguished element $\<0,0,0>$ as an element $\<x,y,z,0>$. We will mark it $\<x,y,z>$.\\
For the element $\<t>=\<t_x,t_y,t_z,0>=\<t_x,t_y,t_z> $
 Eq. (\ref{7}) takes the following form
 
 \begin{align}\begin{split}
  \forall x,y,z,t \in \mathbb{R}  \\
   \<x,y,z,t>&= \< x+t_x,y+t_y,z+t_z> \\
  &= \< x+t_x,y+t_y,t_z> + \<x+t_x,t_y,z+t_z> + \< t_x,y+t_y,z+t_z>  \\
             &\quad - \<x+t_x,t_y,t_z> -\< t_x,y+t_y,t_z> - \< t_x,t_y,z+t_z > \\
             &\quad + \< t_x,t_y,t_z >.    \label{221}
 \end{split}   \end{align}

\noindent It is easy to see that the elements $\<x,y,z>$ can be interpreted as the points of the space $\mathbb{R}^3$.\\
So all triangles from   $\mathbb{R}_2$, satisfying the condition $x+y+z=t$ for a fixed $t$, correspond to a plane in $\mathbb{R}^3$.\\
Let us also observe that the elements of Eq. (\ref{7}) form the vertices of a rectangular cuboid in $\mathbb{R}^3$.\\

\subsection{Multiplication in the set $\mathbb{R}_2$}

\noindent We know that the set $\mathbb{R}_2$ is closed under multiplication. So what does multiplication of elements $\<x,y,z> \in \mathbb{R}_2$ look like? Let us assume that all components of $\<x,y,z> $ are equivalent,  multiplication is associative and commutative, and multiplication is distributive over addition with respect to the components.  Then the following equation must hold.  
\begin{align}
     \Big(\<1,0,0>^2\Big)\<0,1,0>=\<1,0,0>\Big(\<1,0,0>\<0,1,0>\Big).
     										\label{222}
 \end{align}   

\noindent Let us suppose that \\
$\<1,0,0>^2= \<1-2a,a,a>, \qquad \<1,0,0>\<0,1,0>=\<c,c,1-2c>$.\\
Then
\begin{align*}\begin{split}
 \Big(\<1,0,0>^2\Big)\<0,1,0>=&\ \<1-2a,a,a>\<0,1,0>  \\
        = &\ \<1-2a,0,0>\<0,1,0> +\<0,a,0>\<0,1,0> +\<0,0,a>\<0,1,0>\\
        = &\ \<(1-2a)c,(1-2a)c,(1-2a)(1-2c)>\\ 
          &\ +\<a^2,a(1-2a),a^2>\\
          &\ +\<a(1-2c),ac,ac>\\
        = &\ \<a+a^2-4ac+c,a-2a^2-ac+c,a^2-2a+5ac-2c+1>,
 \end{split}   \end{align*}
\noindent and
\begin{align*}\begin{split}
\<1,0,0>\Big(\<1,0,0>\<0,1,0>\Big) =&\ \<1,0,0>\<c,c,1-2c>  \\
        = &\ \<1,0,0>\<c,0,0> +\<1,0,0>\<0,c,0>\\
          &\ +\<1,0,0>\<0,0,1-2c>\\
        = &\ \<(1-2a)c,ac,ac>\\ 
          &\ +\<c^2,c^2,c(1-2c)>\\
          &\ +\<c(1-2c),(1-2c)^2,c(1-2c)>\\
        = &\ \<2c-2ac-c^2,ac-4c+5c^2+1,2c+ac-4c^2>. 
 \end{split}   \end{align*}
 
 \noindent According to Eq. (\ref{222}) we get
 \begin{multline}
 \<a+a^2-4ac+c,a-2a^2-ac+c,a^2-2a+5ac-2c+1>\\
 = \<2c-2ac-c^2,ac-4c+5c^2+1,2c+ac-4c^2>.
     										\label{223}
 \end{multline}  
 \noindent If we compare the first components of the elements in the Eq. (\ref{223}), we get two cases\\
 $a=c=\frac{1}{3} \quad$ or $\quad a=c+1$.\\
  The first case appears to be of little significance, therefore we shall consider the second case.\\
 Let us compare the second components of the elements in the Eq. (\ref{223}) and put $\quad a=c+1$.\\
 We get $\quad a=-\frac{1}{3}, \qquad c= \frac{2}{3}$.\\
 
\noindent It is easy to check that  the obtained values satisfy  Eq.(\ref{223}) for the third components. So
\begin{align*}
     \<1,0,0>^2=\Big\langle\frac{5}{3},-\frac{1}{3},-\frac{1}{3}\Big\rangle, \qquad	
    \<1,0,0>\<0,1,0>=\Big\langle\frac{2}{3},\frac{2}{3},-\frac{1}{3}\Big\rangle.	 		
 \end{align*}   
We assumed that all components of $\<x,y,z> $ are equivalent therefore
\begin{align*}
     \<0,1,0>^2 & =\Big\langle-\frac{1}{3},\frac{5}{3},-\frac{1}{3}\Big\rangle, & \qquad	
     \<0,0,1>^2=\Big\langle-\frac{1}{3},-\frac{1}{3},\frac{5}{3}\Big\rangle, \\[0.2cm]
    \<1,0,0>\<0,0,1> & =\Big\langle\frac{2}{3},-\frac{1}{3},\frac{2}{3}\Big\rangle, & \qquad     
    \<0,1,0>\<0,0,1>=\Big\langle-\frac{1}{3},\frac{2}{3},-\frac{2}{3}\Big\rangle. 
\end{align*} 
It is easy to verify that the following equation holds for the calculated values of a and c.
\begin{align*}
     \Big(\<1,0,0>\<0,1,0>\Big)\<0,0,1>
               =\<1,0,0>\Big(\<0,1,0>\<0,0,1>\Big)
               =\Big\langle\frac{1}{3},\frac{1}{3},\frac{1}{3}\Big\rangle.              \end{align*}  
   
\noindent Let us find a general formula for multiplication for arbitrary $x,y,z,a,b,c \in \mathbb{R}$.
\begin{align}\begin{split}
\<x,y,z>\<a,b,c> =&\
    \Big\langle\frac{5xa}{3},-\frac{xa}{3},-\frac{xa}{3}\Big\rangle
  +\Big\langle\frac{2xb}{3},\frac{2xb}{3},-\frac{xb}{3}\Big\rangle     
  +\Big\langle\frac{2xc}{3},-\frac{xc}{3},\frac{2xc}{3}\Big\rangle\\[0.1cm]
 & +\Big\langle\frac{2ya}{3},\frac{2ya}{3},-\frac{ya}{3}\Big\rangle
  +\Big\langle-\frac{yb}{3},\frac{5yb}{3},-\frac{yb}{3}\Big\rangle
  +\Big\langle-\frac{yc}{3},\frac{2yc}{3},\frac{2yc}{3}\Big\rangle\\[0.1cm]
 & +\Big\langle\frac{2za}{3},-\frac{za}{3},\frac{2za}{3}\Big\rangle
  +\Big\langle-\frac{zb}{3},\frac{2zb}{3},\frac{2zb}{3}\Big\rangle
  +\Big\langle-\frac{zc}{3},-\frac{zc}{3},\frac{5zc}{3}\Big\rangle\\[0.1cm]
     = &\ \Big\langle x(a+b+c)+\frac{(x+y+z)(2a-b-c)}{3},\\
            & \qquad\qquad   y(a+b+c)+\frac{(x+y+z)(-a+2b-c)}{3},\\
       & \qquad \qquad \qquad \qquad z(a+b+c)+\frac{(x+y+z)(-a-b+2c)}{3}\Big\rangle\\[0.1cm]
  = &\ \Big\langle \frac{(5x-y-z)(a+b+c)+(x+y+z)(5a-b-c)}{6},\\
            & \qquad \frac{(-x+5y-z)(a+b+c)+(x+y+z)(-a+5b-c)}{6},\\
   & \qquad \qquad \frac{(-x-y+5z)(a+b+c)+(x+y+z)(-a-b+5c)}{6}\Big\rangle            .  \label{224}
 \end{split}   \end{align} 

\noindent Depending on the context, we will use one of the above multiplication formulas.\\
Now we can verify that the associative property of multiplication holds for any elements of the set $\mathbb{R}_2$.

 \begin{thm}
\begin{align*} 
 \forall x,y,z,a,b,c,A,B,C \in \mathbb{R} \qquad\qquad\\
 \Big(\<x,y,z>\<a,b,c>\Big)\<A,B,C>&\ = \<x,y,z>\Big(\<a,b,c>\<A,B,C>\Big)
 \end{align*}
 \end{thm}
 \begin{proof}
Becouse  all components of $\<x,y,z> $ are equivalent it is enough to verify associativity on first components of $\<x,y,z>$.\\
Let us denote
\[\<x,y,z>\<a,b,c>=\<f,g,h>,\]
where
\begin{align*}
f &= \frac{(5x-y-z)(a+b+c)+(x+y+z)(5a-b-c)}{6}, \\
g &= \frac{(-x+5y-z)(a+b+c)+(x+y+z)(-a+5b-c)}{6},\\
h &= \frac{(-x-y+5z)(a+b+c)+(x+y+z)(-a-b+5c)}{6}.
\end{align*}
Then
\begin{multline*}
\Big(\<x,y,z>\<a,b,c>\Big)\<A,B,C>  = \<f,g,h>\<A,B,C>\\
  = \Big\langle \frac{(5f-g-h)(A+B+C)+(f+g+h)(5A-B-C)}{6},\quad ,\quad                         
     \Big\rangle\\
\end{multline*}
We have 
\begin{align*}\begin{split}
(5f-g-h)& = 5\frac{(5x-y-z)(a+b+c)+(x+y+z)(5a-b-c)}{6}\\
        &\quad - \frac{(-2x+4y+4z)(a+b+c)+(x+y+z)(-2a+4b+4c)}{6}\\ 
        & = (5x-y-z)(a+b+c)+(x+y+z)(4a-2b-2c)  
\end{split}\end{align*}
and
\[f+g+h= (x+y+z)(a+b+c).\]
So the numerator of the first component of the expression 
$\Big(\<x,y,z>\<a,b,c>\Big)\<A,B,C> $ will take the form
\begin{multline*}
(5f-g-h)(A+B+C)+(f+g+h)(5A-B-C)\\
  =\big[(5x-y-z)(a+b+c) +(x+y+z)(4a-2b-2c)\big](A+B+C) \\
     +(x+y+z)(a+b+c)(5A-B-C)
\end{multline*}
If we substitute the expression $(5x-y-z)$ with $(5A-B-C)$, 
$\<x+y+z> $ with $\<A+B+C>$ and vice versa, the above equation remains unchanged. This shows that multiplication is associtive.\\
 \end{proof}

\noindent From Eq. (\ref{224}) we get
 \begin{multline}
 \<x,y,z>^2 =
  \Big\langle \frac{(x+y+z)(5x-y-z)}{3},\\
  \frac{(x+y+z)(-x+5y-z)}{3},\\
      \frac{(x+y+z)(-x-y+5z)}{3}\Big\rangle .
     										\label{226}
 \end{multline}
 \begin{multline*}
\<x,y,z>^3 
  =  \Big\langle \frac{(x+y+z)^2(7x-2y-2z)}{3},\\
               \frac{(x+y+z)^2(-2x+7y-2z)}{3},\\
               \frac{(x+y+z)^2(-2x-2y+7z)}{3}\Big\rangle .     		\end{multline*}   
\begin{multline*}
\<x,y,z>^4 
  =  \< (x+y+z)^3(3x-y-z),\\
               (x+y+z)^3(-x+3y-z),\\
               (x+y+z)^3(-x-y+3z)> .  
\end{multline*}
\begin{align*}
\<x,y,z>\<z,x,y>\<y,z,x> 
  =  \Big\langle \frac{(x+y+z)^3}{3},
                 \frac{(x+y+z)^3}{3},
                 \frac{(x+y+z)^3}{3}\Big\rangle . 
\end{align*}   
From the equation below
\[ \Big\langle \frac{1}{3},\frac{1}{3},\frac{1}{3}\Big\rangle                  
\<x,y,z> = \<x,y,z>\]
it follows that the element $\big\langle \frac{1}{3},\frac{1}{3},\frac{1}{3}\big\rangle $  is the multiplicative identity with respect to the multiplication defined by Eq. (\ref{224}).\\

\noindent As shown in the next equation
\[\<x,y,-x-y>\<a,b,-a-b>=\<0,0,0>,\]
the product of two points is the  point $\<0,0,0>$.\\
But the product of any point different from $\<0,0,0>$ and any triangle 
\[\<x,y,-x-y>\<a,b,c>=\<x(a+b+c),y(a+b+c),(-x-y)(a+b+c)>\]
gives a point different from the point $\<0,0,0>$.

\noindent Let us examine how the components of the triangle $\<1>$ behave under squaring.

\[\forall a,b \in \mathbb{R} \qquad \<1>^2 =
 \< a,b,1-a-b>^2=\Big\langle 2a-\frac{1}{3},2b-\frac{1}{3},
 \frac{5}{3}-a-b\Big\rangle . \]
 We can see that \\
\[\textrm{ if} \quad  a>\frac{1}{3} \quad\textrm{then}\quad
      2a-\frac{1}{3} > a\]
 and vice versa.
 
\noindent Example.
 \[ \Big\langle \frac{1}{2},\frac{1}{3},\frac{1}{6} \Big\rangle^2=
    \Big\langle \frac{2}{3},\frac{1}{3},0 \Big\rangle, \qquad
    \<2,0,-1>^2= 
    \Big\langle \frac{11}{3},-\frac{1}{3},-\frac{7}{3} \Big\rangle.\]
\noindent
A similar situation  occurs in the multiplication of distinct elements $\<1>$. 
 \begin{align*}
  \forall x,y,a,b \in \mathbb{R}  \qquad\qquad\qquad \qquad &  \\
   \<x,y,1-x-y>\<a,b,1-a-b> &=
   \Big\langle x+a-\frac{1}{3},y+b-\frac{1}{3},
 \frac{5}{3}-x-y-a-b\Big\rangle    
    \end{align*}
Here we have.
\[\textrm{ If} \quad x+a>\frac{1}{3} \quad\textrm{then}\quad
      x+a-\frac{1}{3} > \frac{x+a}{2}\]
 and vice versa.\\
 Let us examine how the components of any triangle $\<x,y,z>$ behave under squaring.\\
 According to Eq.(\ref{226})
  it is easy to see that \\
\[\textrm{ if} \quad  \frac{x}{x+y+z}>\frac{1}{3} \quad\textrm{then}\quad
      \frac{(x+y+z)(5x-y-z)}{3} > x(x+y+z)\]
 and vice versa.\\
 
\noindent \subsection{Roots in the set $\mathbb{R}_2$}

\noindent It is straightforward to derive formulas for the roots of elements 
 $\<x,y,z>$, where $x+y+z \neq 0$.\\
 \begin{align*}
\sqrt{\<x,y,z>} & = 
  \Big\langle\frac{4x+y+z}{6\sqrt{x+y+z}},
  \frac{x+4y+z}{6\sqrt{x+y+z}},
  \frac{x+y+4z}{6\sqrt{x+y+z}}\Big\rangle,\\[0.1cm]
\sqrt{\<x,y,1-x-y>} &= 
  \Big\langle\frac{3x+1}{6},\frac{3y+1}{6},
  \frac{4-3x-3y}{6}\Big\rangle,\\[0.1cm]
\<x,y,1-x-y>^\frac{1}{2^n} &=
    \Big\langle\frac{3x+2^n-1}{3\cdot2^n},\frac{3y+2^n-1}{3\cdot2^n},
     \frac{2^n+2-3x-3y}{3\cdot2^n}\Big\rangle,\\[0.1cm] 
\<1,0,0>^\frac{1}{2^n} &=
    \Big\langle\frac{2^n+2}{3\cdot2^n},\frac{2^n-1}{3\cdot2^n},
     \frac{2^n-1}{3\cdot2^n}\Big\rangle.
\end{align*}
Let us propose two additional possible formulas.\\
\begin{align*}
\<x,y,1-x-y>^\frac{1}{n} &=
    \Big\langle\frac{3x+n-1}{3\cdot n},\frac{3y+n-1}{3\cdot n},
     \frac{n+2-3x-3y}{3\cdot n}\Big\rangle,\\[0.1cm] 
\<1,0,0>^\frac{1}{n} &=
    \Big\langle\frac{n+2}{3\cdot n},\frac{n-1}{3\cdot n},
     \frac{n-1}{3\cdot n}\Big\rangle.
\end{align*}
\subsection{Division in the set $\mathbb{R}_2$}

For all 
$  x,y,z,a,b,c \in \mathbb{R}$, such that  $a+b+c\neq 0 $
 there exist $D,E,F \in \mathbb{R} $ such that
\begin{align}
\<x,y,z>= \<D,E,F>\<a,b,c>, \label{227}
\end{align}
 where\\
 \begin{align*}
  D & =\dfrac{x(a+4b+4c)+(y+z)(-2a+b+c)}{3(a+b+c)^2}\\[0.1cm]
    & = \frac{(2x-y-z)(a+b+c)+(x+y+z)(-a+2b+2c)}{3(a+b+c)^2},    
\end{align*}
\begin{align*}
  E & =\dfrac{y(4a+b+4c)+(x+z)(a-2b+c)}{3(a+b+c)^2}\\[0.1cm]
    & = \frac{(-x+2y-z)(a+b+c)+(x+y+z)(2a-b+2c)}{3(a+b+c)^2},    
\end{align*}
\begin{align*}
  F & =\dfrac{z(4a+4b+c)+(x+y)(a+b-2c)}{3(a+b+c)^2}\\[0.1cm]
    & = \frac{(-x-y+2z)(a+b+c)+(x+y+z)(2a+2b-c)}{3(a+b+c)^2}.    
\end{align*}
So if an element $\<a,b,c>$ is not a point, then it divides 
any triangle $\<x,y,z>$.\\

\noindent When $z=-x-y$ and $a+b+c\neq 0$, then
\[\<D,E,F>= \Big\langle\frac{x}{a+b+c},\frac{y}{a+b+c},
     \frac{-x-y}{a+b+c}\Big\rangle.\]
Let us note another formula.\\
 For all 
$  x,y,z,a,b,c \in \mathbb{R}$, such that  $x+y+z\neq 0, \quad a+b+c\neq 0 $\\
 there exist $\alpha,\beta,\gamma \in \mathbb{R} $ such that
\begin{align*}
\<x,y,z>\<a,b,c>\<\alpha,\beta,\gamma>= 
	\Big\langle \frac{1}{3},\frac{1}{3},\frac{1}{3}\Big\rangle,
\end{align*}   
where\\
\begin{multline*}
\alpha = \frac{-xa+(y+z)(b+c)}{(x+y+z)^2(a+b+c)^2},\\
   \beta = \frac{-yb+(x+z)(a+c)}{(x+y+z)^2(a+b+c)^2},\\
      \gamma = \frac{-zc+(x+y)(a+b)}{(x+y+z)^2(a+b+c)^2}.   
\end{multline*}

\section{Dissection of triangles into triangles}

It is known \cite{Dui,Gam} that  a square can be dissected into at least 21 squares of different sides.
It is also known \cite{Tut1}-\cite{DraHam} that a triangle can be dissected into at least 15  similar triangles of different sides.
At the same, it is considered that the triangles, one of which is a mirror image of the other 
are different.\\
To describe the dissection of triangle, we will use the simplified version of the set  $\mathbb{R}_2$.
Let us take the set  $\mathbb{R}_{02}= \{\pm  \<x >=\pm (x^2,x); x\in\mathbb{R} \}$ with the multiplication
\begin{align*}
(x^2_1,x_1)\cdot (x^2_2,x_2)=\big((x_1x_2)^2,(x_1 x_2)\big)
\end{align*}
and addition (\ref{3}). 
The elements of  $\mathbb{R}_{02}$ are interpreted in the same way as the elements of set  $\mathbb{R}_{2}$, as similar triangles. But elements $(x^2,x)$  describes triangles without an edge.
 
\noindent In Fig. \ref{f37} we have one of two known possible optimum dissection of the triangle. \\
\begin{figure}[H]
\begin{center}
 \begin{tikzpicture}[>=stealth,scale=0.28]

\draw(0,0)--(42.9,0)--(21.45,39)--cycle;
\draw(30.8,0)--(36.8,11)--(24.75,11)--cycle;
\draw(20.35,13)--(17.6,8)--(23.1,8)--cycle;
\draw(24.2,20)--(28.05,13)--(20.35,13)--cycle;
\draw(26.95,11)--(31.9,20)--(11,20)--(22,0)--(26.4,8);
\draw(26.95,11)--(25.85,13)--(23.1,8)--(26.4,8);
\draw node at(10.3,8){$\< 20>$};
\draw node at(21,27){$\< 19>$};
\draw node[green] at(30.3,7){$\<-11>$};
\draw node[green] at(36.4,3.5){$\<11>$};

\draw node at(17.4,15.5){$\<-12>$};
\draw node[magenta] at(24,14.8){$\<7>$};
\draw node[magenta] at(27.9,17.2){$\<-7>$};

\draw node at(24.85,8.9){\scriptsize{$\<3>$}};
\draw node[red] at(25.8,11.45){\scriptsize{$\<2>$}};
\draw node[red] at(26.9,12.4){\scriptsize{$\<\text{-}2>$}};

\draw node at(32,14){$\<9>$};
\draw node[cyan] at(23.2,11.4){\scriptsize{$\<-5>$}};
\draw node[cyan] at(20.3,9.4){\scriptsize{$\<5>$}};
\draw node[blue] at(26,2.5){$\<8>$};
\draw node[blue] at(21.5,5){$\<-8>$};
\end{tikzpicture}
\caption{Dissection of the triangle into  15  different  triangles  }
\label{f37}
\end{center}
\end{figure}

This dissection can be written by using 7 times  Eq. (\ref {7}). \\
\begin{align}\begin{split}
  \<39> &= \<19,12,20,-12> \\
  &= \<19,12,0,-12> +\< 0,12,20,12> + \<19,0,20,-12> \\
  &\quad - \<19,0,0,-12> -\<0,12,0,-12> - \<0,0,20,-12>   +\< 0,0,0,-12> \\
   &= \< 19> +\<20> + \< 27> - \cancel{ \<  7>_1} - \cancel{\<8>_2} + \<-12>, \label{23}
\end{split}\end{align}
where
\begin{equation}
    \<27> = \<11,16,11,-11> =
    \<16_1> +\<11> + \<16_2> -  \cancel{\<5>_3}  + \<-11>,  \label{24}
\end{equation}
where
\begin{equation}
 \<16_1> = \< 7,7,9,-7>
  = \cancel{\<7>_1}+\< 9_1> + \< 9_2> -  \cancel{\<2>_4} + \< -7>, \label{25}
\end{equation}
where
\begin{equation}
  \<9_2> = \< 2,7,2,-2>
   = \<7_1> +\cancel{\<2>_4} + \<7_2> -  \cancel{\<5>_5} + \<-2>, \label{26}
\end{equation}
where
\begin{equation}
  \<7_2> = \<5,5,2,-5>
   = \cancel{\<5>_5} +\<2_1> + \cancel{\<2_2>_7} -
   \cancel{\< -3>_6} + \< -5>, \label{27}
\end{equation}
where
\begin{equation}
 \<16_2> = \<8,8,8,-8>
  = \< 8_1> +\<8_2> + \cancel{\<8_3>_2}  + \<-8>, \label{28}
  \end{equation}
where
  \begin{equation}
 \<8_1> = \<3,5,3,-3>
  = \cancel{\<5_1>_3} +\<5_2> + \< 3> -  \cancel{\<2>_7} +
  \cancel{\<-3>_6}. \label{29}
\end{equation}

\noindent It should be noted that the triangles in order to be reduced have to be of different signs and lie one on the other. The strikethrough triangle $- \cancel{ \<7>_1}$  in Eq. (\ref{23}) means that it reduces with the strikethrough triangle $ \cancel{ \<7>_1}$ in Eq. (\ref{25}). The other strikethrough triangles mean the same.\\
We can see Eqs.(\ref{23})-(\ref{29}) on Fig.~\ref{f38}.\\
The blue triangle $-\<7>$ in the black triangle $\<39>$ means that it reduces with the black triangle $\<7>$ in the blue triangle $\<16>$. Similarly  reduce other triangles.\\
The gray number  $\<27>$ means that the triangle  $\<27>$ will be replaced with a distribution into smaller triangles $ \<27> = \<11,16,11,-11>$. It is similar with the other gray numbers.\\
\begin{figure}[H]
\begin{center}
 \begin{tikzpicture}[>=stealth,scale=0.25]

\draw(2,36.98)--(41,36.98)--(21.5,70.52)--cycle;
\draw(22,36.98)--(12,54.18)--(31,54.18);
\draw(14,36.98)--(27.5,60.2);
\node at(12,42){$\<20>$};\node at(21,60){$\<19>$};
\node at(18,50){$\<-12>$};\node[red] at(17.8,39){$-\<8>$};
\node[cyan] at(27.2,56.5){$-\<7>$};\node[gray] at(27,44){$\<27>$};
\node at(14.8,67){$\< 39>=\<19,12,20,-12>$};

\draw[green](8.5,27.52)--(16.5,13.76)--(22,23.22)--(6,23.22)--cycle;
\draw[green](0.5,13.76)--(27.5,13.76)--(14,36.98)--cycle;
\draw[->](17,33)--(19.1,36.6);
\node[green] at(16,19.8){$\<-11>$};
\node[green] at(21.5,16.5){$\<11>$};
\node[gray] at(13.5,28){$\<16>$};
\node[gray] at(8.7,18){$\<16>$};
\node[orange] at(8.3,24.3){\scriptsize{$-\<5>$}};
\node[green] at(8,33){\small{$\<27>=\<11,16,11,-11>$}};

\draw[red](8.5,0)--(24.5,0)--(16.5,13.76)--cycle;
\draw[red](12.5,6.88)--(20.5,6.88)--(16.5,0)--cycle;
\draw[->](14.5,11)--(13.2,13.3);
\node at(12.5,2.7){$\<8>$};\node[red] at(16.6,4){$\<-8>$};
\node [gray] at(16.5,8.8){$\<8>$};
\node[red] at(20.7,2.7){$\<8>$};
\node[red] at(5.3,0.7){\small{$\< 16>=\< 8,8,8,-8>$}};

\draw[orange](4.5,6.88)--(12.5,6.88)--(8.5,13.76)--cycle;
\draw[orange](6,9.46)--(11,9.46)--(9.5,6.88)--(7,11.18);
\draw[->](11.9,9)--(13.3,9);
\node[blue] at(6.7,9.9){\scriptsize{$\text{-}\<2>$}};
\node[green] at(8.9,10.7){\scriptsize{$\< 5>$}};
\node [orange] at(6.9,7.8){\scriptsize{$\< 5>$}};
\node[blue] at(9.4,8.65){\scriptsize{$\<\text{-}3>$}};
\node[orange] at(11,7.7){\scriptsize{$\<3>$}};
\node[orange] at(3.7,12.4){\small{$\< 8> =\< 3,5,3,-3>$}};

\draw[cyan](22,23.22)--(38,23.22)--(30,36.98)--cycle;
\draw[cyan](29,23.22)--(33.5,30.96)--(26.5,30.96)--(31,23.22)--cycle;
\draw[->](24.8,29)--(19.1,29);
\node[gray] at(26.5,25.5){$\< 9>$};
\node[cyan] at(29.9,28.8){$\<-7>$};
\node[cyan] at(33.5,25.5){$\<9>$};
\node at(30,32.5){$\<7>$};
\node[magenta] at(29.65,23.7){\scriptsize{-$\<2>$}};
\node[cyan] at(34.7,35){\small{$\<16>=\<7,7,9,-7>$}};

\draw[magenta](26.5,15.48)--(35.5,15.48)--(31,23.22)--cycle;
\draw[magenta](27.5,17.2)--(34.5,17.2)--(33.5,15.48)--(30,21.5)--cycle;
\draw[->](28.2,19)--(25.9,23);
\node[gray] at(29.8,16.2){\scriptsize{$\<7>$}};
\node[blue] at(29.7,18.5){\scriptsize{-$\<5>$}};
\node[magenta] at(32.4,18.9){\scriptsize{$\<7>$}};
\node[magenta] at(33.1,16.55){\scriptsize{$\<\text{-}2>$}};
\node[cyan] at(34.5,15.9){\scriptsize{$\<2>$}};
\node[magenta] at(35.5,21.5){\small{$\<9>=\<2,7,2,-2>$}};

\draw[blue](30,9.46)--(37,9.46)--(33.5,15.48)--cycle;
\draw[blue](31,11.18)--(36,11.18)--(33.5,6.88)--cycle;
\draw[->](31.7,13.1)--(30.3,15.3);
\node[orange] at(30.9,9.9){\scriptsize{$\< 2>$}};
\node[magenta] at(33.4,12.2){\scriptsize{$\< 5>$}};
\node[blue] at(33.4,10.2){\scriptsize{$\<\text{-}5>$}};
\node[orange] at(33.5,8.75){\scriptsize{-$\<\text{-}3>$}};
\node[blue] at(35.9,9.9){\scriptsize{$\<2>$}};
\node[blue] at(38.2,14){\small{$\<7>=\<5,5,2,-5>$}};
\end{tikzpicture}
\caption{Scheme  of a reduction  of triangles   }
\label{f38}
\end{center}
\end{figure}

\section{The set  $\mathbb{R}_1$.}

\noindent The set $\mathbb{R}_2$  has its counterpart on the number line.
Let us take the ring\\
 $\mathbb{R}^2=\mathbb{R}\times\mathbb{R}= \{ (x,y) ; \; x,y\in\mathbb{R} \}$ with  additon and multiplication\\
\begin{equation}
(x_1,y_1)+(x_2,y_2)=(x_1+x_2,y_1+y_2), \label{36}
\end{equation}
\begin{equation}
(x_1,y_1)\cdot (x_2,y_2)=(x_1\cdot x_2,y_1\cdot y_2). \label{37}
\end{equation}
Let us consider the subset of the ring $\mathbb{R}^2$, the set $\mathbb{R}_1= \{\pm  \<x >=\pm (x,1); x\in\mathbb{R} \}$.\\
It is closed under multiplication (\ref{37}) but not under addition (\ref{36}).\\
The set $\mathbb{R}_1$ is closed under    the following kind of addition  
 \begin{align}
\forall x,y,t \in \mathbb{R}\qquad
       \<x+y+t> = &\  \< x+t> + \<y+t> - \<t> \qquad
    \label{38}
\end{align} 
because the equations
\begin{align*}
\forall x,y,t \in \mathbb{R}\quad \forall i=1,0\qquad \qquad \\
       (x+y+t)^i = &\  ( x+t )^i +  ( y+t )^i - (t)^i  
   \end{align*}
are true.\\
Later in this article, an element $\<x>$ can be replaced with $-\<x>$ and vice versa.\\

\noindent We can check  that for each $x \in \mathbb{R}$ the following equation holds.
\begin{align}
\forall x \in \mathbb{R}\qquad
       \<x>= x\<1>-(x-1)\<0>.
    \label{39}
\end{align} 
Let us transform Eq. (\ref{39}). 
\begin{align*}
\< x > = &\ x\<1>-(x-1)\<0> =  x\big(\<1>-\<0>\big) + 1\<0>.
\end{align*}
 It is easy to check that 
 the elements  $\< 1> -\<0> = A_1$,  $\<0>=A_0$ are orthogonal. So
 $A_1=(1,0)$, \; $A_0=(0,1)$.\\

\noindent An element $\<x>$ for each $x \in \mathbb{R},\: x>0$  will be interpreted as a one-dimensional vector and marked with a black color  
 \tikz{\draw[->] (0,0)--(1,0);
\draw[fill](0,0)circle(1pt);} \ .\\
An element $-\<x>$ will be interpreted as a one-dimensional vector with a direction opposite to the direction of vector $-\<x>$ and marked with a red color
\tikz{\draw[red,<-] (0,0)--(1,0);
\draw[fill,red](1,0)circle(1pt);} \;.\\  
An element $\<0>$ will be marked as a black point and an element $-\<0>$ will be marked as a red point.\\
To see the result of putting the black triangle on the red one, we will denote the empty set of green.\\  
An element $-\<-x>$ for each $x \in \mathbb{R},\: x>0$  will be interpreted as a one-dimensional vector but without endpoints and marked with a black color with green endpoints 
 \tikz{\draw[] (0,0)--(.9,0);
\draw[fill,green](0,0)circle(1pt);
\draw[green,->](.9,0)--(1,0)} \ .\\
An element $\<-x>$ for each $x \in \mathbb{R},\: x>0$  will be interpreted as a one-dimensional vector  without endpoints and marked with a red color with green endpoints \;
 \tikz{\draw[red] (.1,0)--(1,0);
\draw[fill,green](1,0)circle(1pt);
\draw[green,<-](0,0)--(.1,0)} \ .\\

\noindent Why the vectors $\<-x>$ and $-\<-x>$ for $ x>0$  are without the endpoints we will see after introducing the geometric interpretation of the operation (\ref{38}).\\

\noindent We will  create  a geometric construction of the addition (\ref{38}).\\
 The sum $ \<x+t> $ is obtained by placing  the tail of the vector $ \<x> $ at the head of the vector $\<t> $ (Fig. \ref{f41}).  The sum $ \<y+t> $ is obtained by placing  the tail of the vector $ \<t> $ at the head of the vector $ \<y> $.\\
However, it should be emphasized that each of the elements  $ \<x+t> $  and  $ \<y+t> $  should be treated as one vector. The division  of the vector  $ \<x+t> $ into the two vectors  $ \<t> $ and  $ \<x> $ shown in the  Fig. \ref{f41} is only to illustrate where the vector  $ \<x+t> $  lies versus the vector  $ \<t> $ .

\begin{figure}[H]
\begin{center}
 \begin{tikzpicture}[>=stealth]
\draw[->](0,2.2)--(1,2.2);\draw[fill](0,2.2)circle(1pt);\node at(3.1,2.1){$\< t >$, $ t>0 $};
\draw[->](1,1.2)--(2,1.2);\draw[fill](1,1.2)circle(1pt);\node at(3.1,1.1){$\< x >$, $ x>0 $};
\draw[->](0,0.2)--(2,0.2);\draw[fill](0,.2)circle(1pt);\node at(3.1,0.1){$\< x+t >$};
\draw[red](5.05,2.2)--(5.9,2.2);\draw[green,<-](4.9,2.2)--(5.05,2.2);
  \draw[fill,green](5.9,2.2)circle(1pt); \node at(8,2.1){$\< t >$, $ t<0 $};
\draw[->](4.9,1.2)--(6.9,1.2);\draw[fill](4.9,1.2)circle(1pt);\node at(8,1.1){$\< x >$, $ x>0 $};
\draw[->](5.9,0.2)--(6.9,0.2);\draw[fill](5.9,.2)circle(1pt);\node at(8,0.1){$\< x+t >$};
\draw[red](9.75,2.2)--(11.6,2.2);\draw[green,<-](9.6,2.2)--(9.75,2.2);
\draw[fill,green](11.6,2.2)circle(1pt);\node at(12.7,2.1){$\< t >$, $ t<0 $};
\draw[->](9.6,1.2)--(10.6,1.2);\draw[fill](9.6,1.2)circle(1pt);\node at(12.7,1.1){$\< x >$, $ x>0 $};
\draw[red](10.75,0.2)--(11.6,0.2);\draw[green,<-](10.6,0.2)--(10.75,0.2);
  \draw[fill,green](11.6,0.2)circle(1pt);\node at(12.7,0.1){$\< x+t >$};
\end{tikzpicture}
\caption{ The sum $ \<x+t> $ for different $ t $.}
\label{f41}
\end{center}
\end{figure}
\noindent Figure (\ref{f42}) illustrates the geometric interpretation of Eq. (\ref{38}) for  positive real numbers $ x,y,t $.
\begin{figure}[H]
\begin{center}
 \begin{tikzpicture}[>=stealth]
\draw[->](1,3.2)--(2,3.2);\draw[fill](1,3.2)circle(1pt);\node[right] at(3.2,3.1){$\< t >$};
\draw[->](1,2.2)--(3,2.2);\draw[fill](1,2.2)circle(1pt);\node[right] at(3.2,2.1){$\<x+ t >$};
\draw[->](0,1.5)--(2,1.5);\draw[fill](0,1.5)circle(1pt);\node[right] at(3.2,1.4){$\< y+t >$};
\draw[red,<-](1,0.8)--(2,0.8);\draw[fill,red](2,.8)circle(1pt);\node[right] at(3.2,0.8){$-\<t >$};
\draw[->](0,0.1)--(3,0.1);\draw[fill](0,.1)circle(1pt);\node[right] at(3.2,0){$\< x+y+t >$};
\end{tikzpicture}
\caption{ The geometric interpretation of Eq. (\ref{38}) for  positive real numbers $ x,y,t $.}
\label{f42}
\end{center}
\end{figure}

\noindent As we can see, the order of components in the element $\<x+y+t>$ is important, therefore Eq. (\ref{38}) should be written as follows
\begin{align}
\forall x,y,t \in \mathbb{R}\qquad
       \<x+y+t> = &\  \< x+0+t> + \<0+y+t> - \<0+0+t> \qquad
       \label{382}
    \end{align} 

\noindent From the below equation
\[
\< -1>=\<-1-1+1>=\<-1+0+1>+\<0-1+1>-\<1>=2\<0>-\<1>
\]
we  can get the way  of building of the triangle $\<-1>$.
Fig. \ref{f43} shows each step of  the construction of $\<-1>$.
\begin{figure}[H]
\begin{center}
 \begin{tikzpicture}[>=stealth]
\draw[->](0,2.5)--(1,2.5);\draw[fill](0,2.5)circle(1pt);\node[right] at(1,2.5){$\< 1 >$};
\draw[fill](0,1.5)circle(1pt);\draw[fill](1,1.5)circle(1pt);
\node[left] at(0,1.5){$\<0 >$};\node[right] at(1,1.5){$\<0 >$};
\draw[red,<-](0,0.8)--(1,0.8);\draw[fill,red](1,.8)circle(1pt);\node[right] at(1,0.8){$-\<1 >$};
\draw[red](0.15,0.1)--(1,0.1);\draw[fill,green](1,.1)circle(1pt);
\draw[green,<-](0,0.1)--(0.15,.1);\node[right] at(1,0){$\<-1 >$};
\end{tikzpicture}
\caption{ Stages of  the construction of the element $\<-1>$.}
\label{f43}
\end{center}
\end{figure} 
  
\noindent The definitions that an equation is true in the geometric or algebraic sense  are the same as for the elements of the set $\mathbb{R}_2$.  
 
\noindent Because elements of the set $\mathbb{R}_1$  are vectors so Eq. (\ref{382}) is true in the geometric sense $\forall x,y,t \in \mathbb{R}$.  

\noindent According to proof of  Theorem \ref{t2.4},  to show the geometric interpretation of Eq. (\ref{382}) for any real numbers $x,y,t $ it is enough to show it for positive  $ x,y $ and any $ t $.
We have the following cases.
\begin{enumerate}
\item[(1)] $t >0$, (Fig.~\ref{f42}).\\[.1cm]
\text{In next cases  $t<0$}.
\item[(2)] $x+t>0, y+t>0 $, (Fig.~\ref{f44}).
\item[(3)] $x+t>0, y+t<0$, (Fig.~\ref{f44}).
\item[(4)] $x+t<0, y+t<0, x+y+t>0$.
\item[(5)] $x+t<0, y+t<0, x+y+t<0$.
\end{enumerate}

\begin{figure}[H]
\begin{center}
 \begin{tikzpicture}[>=stealth]
\draw[red](1.15,3.2)--(2,3.2);\draw[green,<-](1,3.2)--(1.15,3.2);
   \draw[fill,green](2,3.2)circle(1pt);\node[right] at(3.2,3.1){$\< t >$};
\draw[->](2,2.2)--(3,2.2);\draw[fill](2,2.2)circle(1pt);\node[right] at(3.2,2.1){$\<x+ t >$};
\draw[->](0,1.5)--(1,1.5);\draw[fill](0,1.5)circle(1pt);\node[right] at(3.2,1.4){$\< y+t >$};
\draw(1,0.8)--(1.85,0.8); \draw[green,->](1.85,0.8)--(2,0.8);
    \draw[fill,green](1,.8)circle(1pt);\node[right] at(3.2,0.8){$-\<t >$};
\draw[->](0,0.1)--(3,0.1);\draw[fill](0,.1)circle(1pt);\node[right] at(3.2,0){$\< x+y+t >$};

\draw[red](7.15,3.2)--(9,3.2); \draw[green,<-](7,3.2)--(7.15,3.2);
    \draw[fill,green](9,3.2)circle(1pt);\node[right] at(10.2,3.1){$\< t >$};
\draw[->](9,2.2)--(10,2.2);\draw[fill](9,2.2)circle(1pt);\node[right] at(10.2,2.1){$\<x+ t >$};
\draw[red](7.15,1.5)--(8,1.5); \draw[green,<-](7,1.5)--(7.15,1.5);
   \draw[fill,green](8,1.5)circle(1pt);\node[right] at(10.2,1.4){$\< y+t >$};
\draw(7,0.8)--(8.85,0.8); \draw[green,->](8.85,0.8)--(9,0.8);
    \draw[fill,green](7,.8)circle(1pt);\node[right] at(10.2,0.8){$-\<t >$};
\draw[->](8,0.1)--(10,0.1);\draw[fill](8,.1)circle(1pt);\node[right] at(10.2,0){$\< x+y+t >$};
\end{tikzpicture}
\caption{ The geometric interpretation of Eq. (\ref{382}). The cases (2) and (3).}
\label{f44}
\end{center}
\end{figure}
\noindent But if we replace Eq. (\ref{382}) by  equivalent Eq. (\ref{40})
\begin{align} \begin {split}
 \<t> &= \<-x-y+(x+y+t)>\\
       & = \<-x+0+(x+y+t)>  + \<0-y+(x+y+t)>  - \<0+0+ (x+y+t)> 
    \label{40}
\end {split} \end{align}
 then
$\<t>$ acts as $\< t'> = \<(x+y+t)>$,  $\<x+t> = \<0-y+(x+y+t)>$ acts as $\<0+y+t>$ and $\<y+t> = \<-x+0+(x+y+t)>$ acts as $\<x+0+t>$.\\
\noindent In the case (4) only $t'$ is positive and in the case (2) only $t$ is negative. By changing the sign in all components of the case (4) we will receive the case (2).\\
\noindent Similarly, in the case (5)  all components are negative and in the case (1) all components  are positive. By changing the sign in all components of the case (5) we will receive the case (1).\\
So we have 3 different cases of adding in the set $ \mathbb {R} _1 $. 
 
\noindent It can be proved by mathematical induction that the  Eq. (\ref{39}) follows   from   Eq. (\ref{38}) for natural $n$ (Fig.~\ref{f45}).  
 \begin{align}
\forall n \in \mathbb{N}\qquad
\<n>= n\<1>-(n-1)\big(\<0>\big) = n\<1>+(n-1)\big(-\<0>\big) .
           \label{41}
\end{align}
\begin{figure}[H]
\begin{center}
 \begin{tikzpicture}
\draw[->](7,3)--(8,3);\draw[fill](7,3)circle(1pt);
\draw[->](8,2.7)--(9,2.7);\draw[fill](8,2.7)circle(1pt);
 \draw[->](9,2.4)--(10,2.4);\draw[fill](9,2.4)circle(1pt);
\draw[dashed](10,2.1)--(11,2.1);
\draw[->](11,1.8)--(12,1.8);\draw[fill](11,1.8)circle(1pt);
\node at(5.9,2.3){=};
\draw(0,2.4)--(3,2.4);\draw[fill](0,2.4)circle(1pt);
 \draw[dashed](3,2.4)--(4,2.4);
\draw[->](4,2.4)--(5,2.4);
\draw[fill,red](8,2.4)circle(1pt);\draw[fill,red](9,2.1)circle(1pt);
\draw[fill,red](10,1.8)circle(1pt);\draw[fill,red](11,1.5)circle(1pt);
\node at(2.8,0.8){$\< n >$};
\node at(5.9,0.8){=};
\node at(9,0.8){$n\< 1 >+(n-1)(-\< 0 >)$};
\end{tikzpicture}
\caption{ Interpretation of the equation (\ref{41}).}
\label{f45}
\end{center}
\end{figure}

\section{Multiplication in the set $\mathbb{R}_1$.}

\noindent Let us fix a point in $\mathbb{R}_1$ and denote it as $\<0+0>$. Then, using Eq. (\ref{7}), every element of the set $\mathbb{R}_2$ can be obtained from the distinguished element $\<0+0>$ as an element $\<x+y+0>$. We will mark it $\<x,y>$.\\
For the element $\<t>=\<t_x+t_y+t_y,0>=\<t_x,t_y> $
 Eq. (\ref{382}) takes the following form
 
 \begin{align}\begin{split}
  \forall x,y,t \in \mathbb{R}  \\
   \<x+y+t>&= \< x+t_x,y+t_y> 
  = \< x+t_x,t_y> + \<t_x,y+t_y> - \< t_x,t_y>  \\
             \label{383}
 \end{split}   \end{align}

\noindent It is easy to see that the elements $\<x,y>$ can be interpreted as the points of the space $\mathbb{R}^2$.\\
So all vectors from   $\mathbb{R}_1$ satisfying the condition $x+y=t$ for a fixed $t$, correspond to a line in $\mathbb{R}^2$.\\
Let us also observe that the elements of Eq. (\ref{382}) form the vertices of a rectangle  in $\mathbb{R}^2$.\\

\subsection{Multiplication in the set $\mathbb{R}_1$}

\noindent We know that the set $\mathbb{R}_1$ is closed under multiplication. So what does multiplication of elements $\<x,y> \in \mathbb{R}_1$ look like? Let us assume that all components of $\<x,y> $ are equivalent,  multiplication is associative and commutative, and multiplication is distributive over addition with respect to the components.  Then the following equation must hold.  
\begin{align}
     \Big(\<1,0>^2\Big)\<0,1>=\<1,0>\Big(\<1,0>\<0,1>\Big).
     										\label{384}
 \end{align}

\noindent We assumed that all components of $\<x,y>$ are equivalent so\\
$\<1,0>\<0,1>=\big\langle\frac{1}{2},\frac{1}{2}\big\rangle$.\\[0.1cm]
Let us suppose that \\
$\<1,0>^2= \<1-a,a> $ and $\<0,1>^2= \<a,1-a> $.\\
Then
\begin{align*}\begin{split}
 \Big(\<1,0>^2\Big)\<0,1>=&\ \<a,1-a>\<0,1>  \\
        = &\ \<a,0>\<0,1> +\<0,1-a>\<0,1> \\
        = &\  \Big\langle\frac{a}{2},\frac{a}{2}\Big\rangle
            + \<(1-a)^2,a-a^2> \\[0.1cm]
        = &\ \Big\langle\frac{2-3a+2a^2}{2},\frac{3a-2a^2}{2}\Big\rangle,
 \end{split}   \end{align*}
\noindent and
\begin{align*}\begin{split}
\<1,0>\Big(\<1,0>\<0,1>\Big)
       =&\ \<1,0>\Big\langle\frac{1}{2},\frac{1}{2}\Big\rangle  \\[0.1cm]
        = &\ \<1,0>\Big\langle\frac{1}{2},0\Big\rangle 
            +\<1,0>\Big\langle0,\frac{1}{2}\Big\rangle \\[0.1cm]
        = &\ \Big\langle\frac{a}{2},\frac{1-a}{2}\Big\rangle 
            +\Big\langle\frac{1}{4},\frac{1}{4}\Big\rangle\\[0.1cm]
        = &\ \Big\langle\frac{2a+1}{4},\frac{3-2a}{4}\Big\rangle. 
 \end{split}   \end{align*}
 
 \noindent According to Eq. (\ref{384}) we get
 \begin{align}
 \Big\langle\frac{2-3a+2a^2}{2},\frac{3a-2a^2}{2}\Big\rangle
 =  \Big\langle\frac{2a+1}{4},\frac{3-2a}{4}\Big\rangle.
     										\label{385}
 \end{align}  
 \noindent If we compare the first or second components of the elements in the Eq. (\ref{385}), we get two cases\\
 $a=\frac{1}{2} \quad$ or $\quad a=\frac{3}{2}$.\\
  The first case appears to be of little significance, therefore we shall consider the second case.\\
\noindent So
\begin{align*}
     \<1,0>^2=\Big\langle\frac{3}{2},-\frac{1}{2}\Big\rangle, \qquad	
    \<0,1>^2=\Big\langle-\frac{1}{2},\frac{3}{2}\Big\rangle.	 		
 \end{align*}   
    
\noindent Let us find a general formula for multiplication for arbitrary $x,y,a,b \in \mathbb{R}$.
\begin{align}\begin{split}
\<x,y>\<a,b> =&\
    \Big\langle\frac{3xa}{2},-\frac{xa}{2}\Big\rangle
  +\Big\langle\frac{xb}{2},\frac{xb}{2}\Big\rangle \\[0.1cm]
 & +\Big\langle\frac{a}{2},\frac{ya}{2}\Big\rangle
  +\Big\langle-\frac{yb}{2},\frac{3yb}{2}\Big\rangle  \\[0.1cm]
     = &\ \Big\langle  \frac{x(3a+b)+y(a-b)}{2},
      \frac{x(-a-b)+y(a+3b)}{2}\Big\rangle\\[0.1cm]
  = &\ \Big\langle \frac{(3x-y)(a+b)+(x+y)(3a-b)}{4},\\
     & \qquad \qquad  \frac{(-x+3y)(a+b)+(x+y)(-a+3b)}{4}  \Big\rangle            .  \label{386}
 \end{split}   \end{align}   
 
\noindent Depending on the context, we will use one of the above multiplication formulas.\\

\noindent Now we can verify that the associative property of multiplication holds for any elements of the set $\mathbb{R}_2$.

 \begin{thm}
\begin{align*} 
 \forall x,y,a,b,A,B \in \mathbb{R} \qquad\qquad\\
 \Big(\<x,y>\<a,b>\Big)\<A,B>&\ = \<x,y>\Big(\<a,b>\<A,B>\Big)
 \end{align*}
 \end{thm}
 \begin{proof}
Becouse  all components of $\<x,y> $ are equivalent it is enough to verify associativity on first components of $\<x,y>$.\\
Let us denote
\[\<x,y>\<a,b>=\<f,g>,\]
where
\begin{align*}
f &= \frac{(3x-y)(a+b)+(x+y)(5a-b)}{4}\\[0.1cm]
g  &=  \frac{(-x+3y)(a+b)+(x+y)(-a+3b)}{4}.  
\end{align*}
Then
\begin{multline*}
\Big(\<x,y>\<a,b>\Big)\<A,B>  = \<f,g>\<A,B>\\
  = \Big\langle \frac{(3f-g)(A+B)+(f+g)(3A-B)}{4},\quad                      
     \Big\rangle\\
\end{multline*}
We have
\begin{align*}\begin{split}
(3f-g)& = 3\frac{(3x-y)(a+b)+(x+y)(3a-b)}{4}\\
       & \quad -\frac{(-x+3y)(a+b)+(x+y)(-a+3b)}{4}\\ 
        & = (3x-y)(a+b)+(x+y)(2a-2b)  
\end{split}\end{align*}
and
\[f+g= (x+y)(a+b)\]
So the numerator of the first component of the expression 
$\Big(\<x,y>\<a,b>\Big)\<A,B> $ will take the form
\begin{multline*}
(3f-g)(A+B)+(f+g)(3A-B)\\
  =\big[(3x-y)(a+b)+(x+y)(2a-2b)\big](A+B) \\
     +(x+y)(a+b)(3A-B)
\end{multline*}
If we substitute the expression $(3x-y)$ with $(3A-B)$, 
$\<x+y> $ with $\<A+B>$ and vice versa, the above equation remains unchanged. This shows that multiplication is associtive.\\
 \end{proof}
 From Eq. (\ref{386}) we get
 \begin{align}
 \<x,y>^2 =
  \Big\langle \frac{(x+y)(3x-y)}{2},
  \frac{(x+y)(-x+3y)}{2} \Big\rangle .
     				\label{388}
 \end{align} 
\begin{align*}
\<x,y>\<y,x> 
  =  \Big\langle \frac{(x+y)^2}{2},
                 \frac{(x+y)^2}{2} \Big\rangle . 
\end{align*}   
 
\noindent \subsection{Roots in the set $\mathbb{R}_1$}

\noindent It is straightforward to derive formulas for the roots of elements 
 $\<x,y>$, where $x+y \neq 0$.\\
 \begin{align*}
\sqrt{\<x,y>} & = 
  \Big\langle\frac{3x+y}{4\sqrt{x+y}},
  \frac{x+3y}{4\sqrt{x+y}},\Big \rangle,\\[0.1cm]
\sqrt{\<x,1-x>} &= 
  \Big\langle\frac{2x+1}{4}, \frac{3-2x}{4}\Big\rangle.
\end{align*}

\subsection{Division in the set $\mathbb{R}_1$}

For all 
$  x,y,a,b \in \mathbb{R}$, such that  $a+b\neq 0 $
 there exist $D,E \in \mathbb{R} $ such that
\begin{align}
\<x,y>= \<D,E>\<a,b>, \label{389}
\end{align}
 where\\
 \begin{align*}
\<D,E>&=\Big\langle \frac{x(a+3b)+y(-a+b)}{2(a+b)^2}, 
             \frac{x(a-b)+y(3a+b)}{2(a+b)^2}\Big\rangle\\[0.1cm]\\ 
      &=\Big\langle \frac{(x+y)(-a+3b)+(3x-y)(a+b)}{4(a+b)^2},\\
        &\qquad\qquad\qquad
        \frac{(x+y)(3a-b)+(-x+3y)(a+b)}{4(a+b)^2}\Big\rangle.\\[0.1cm]      
\end{align*}
So if an element $\<a,b>$ is not a point, then it divides 
any vector $\<x,y>$.\\

\noindent When $y=-x$ and $a+b\neq 0$, then
\[\<D,E>= \Big\langle\frac{x}{a+b},\frac{-x}{a+b}\Big\rangle.\]

\section{Other geometric interpretation of the set $\mathbb{R}_1$.} 

\noindent The set $\mathbb{R}_1$ has  another richer geometric interpretation.\\
An element $\<1>$  can be interpreted as a closed triangle  without one vertex and marked with a black color 
\tikz{\draw[fill,gray!60] (0,0)--(1,0)--(.5,.8)--(0,0);
\draw (0,0)--(1,0)--(.5,.8)--(0,0);
\draw[fill](0,0)circle(1pt);\draw[fill](1,0)circle(1pt);
\draw[fill,green](.5,.8)circle(1pt);} \ .\\
An element $\<0>$  can be interpreted as a  triangle  
\tikz{\draw[fill,red!30] (.5,0)--(1,.8)--(0,.8)--(.5,0);
\draw[thick,green] (0,.8)--(.5,0)--(1,.8)--cycle;
\draw[fill](.5,0)circle(1pt);\draw[fill,green](1,.8)circle(1pt);
\draw[fill,green](0,.8)circle(1pt);} .\\
 
\noindent The Figures ~\ref{f46}-~\ref{f48}  show the iterpretation  of $\<x>$ for different $x$. 
\begin{figure}[H]
\begin{minipage}[t]{.45\textwidth}
\begin{center} 
\begin{tikzpicture}[>=stealth]
\draw[fill,gray!60](0,0)--(3,0)--(2.5,.8)--(.5,.8)--cycle;
\draw(.5,.8)--(0,0)--(3,0)--(2.5,.8);
\draw[thick,green](.5,.8)--(2.5,.8);
\draw[fill,green](.5,.8)circle(1pt);\draw[fill,green](2.5,.8)circle(1pt);
\draw [decorate,decoration={brace,amplitude=5pt},
       xshift=0pt,yshift=-3pt](3,0)--(0,0);
\node at(1.5,-.5){$x$};   
\end{tikzpicture}
\caption{$ \<x>$ for $x>1$.}
\label{f46}
\end{center}
\end{minipage}
\hfill
\begin{minipage}[t]{.45\textwidth}
\begin{center} 
\begin{tikzpicture}[>=stealth]
\draw[fill,red!30] (0,.8)--(.5,0)--(2.5,0)--(3,.8)--cycle;
\draw[red](.5,0)--(2.5,0);
\draw[thick,green](2.5,0)--(3,.8)--(0,.8)--(.5,0);
\draw[fill,green](0,.8)circle(1pt); \draw[fill,green](.5,0)circle(1pt); 
\draw[fill,green](3,.8)circle(1pt);\draw[fill,green](2.5,0)circle(1pt); 
\draw [decorate,decoration={brace,amplitude=5pt},
       xshift=0pt,yshift=-3pt](2.5,0)--(.5,0);
\node at(1.5,-.5){$x$};
\end{tikzpicture}
\caption{$\<x>$ for $x<0$.}
\label{f47}
\end{center}
\end{minipage}
\end{figure}

\begin{figure}[H]
\begin{center} 
\begin{tikzpicture}[scale=1.36]
\draw[fill,red!30](.5,.4)--(1,1.2)--(0,1.2)--cycle;
\draw[thick,green](1,1.2)--(.5,.4)--(0,1.2)--cycle;
\draw[fill,green](0,1.2)circle(1pt);\draw[fill,green](1,1.2)circle(1pt);
\draw[fill,gray!60](.25,0)--(.75,0)--(.5,.4)--cycle;
\draw(.25,0)--(.75,0)--(.5,.4)--cycle;
\draw[fill](.5,.4)circle(1pt);
\draw [decorate,decoration={brace,amplitude=5pt},
       xshift=0pt,yshift=-3pt](.75,0)--(.25,0);
\node at(.5,-.35){$x$};
\draw [decorate,decoration={brace,amplitude=5pt},
       xshift=0pt,yshift=3pt](0,1.2)--(1,1.2);
\node at(.5,1.6){$1-x$};
\end{tikzpicture}
\caption{$\<x>$ for $0< x<1$. The Figure is to a scale enlarged twice. }
\label{f48}
\end{center}
\end{figure} 

\noindent In Figures ~\ref{f49}-~\ref{f54} we can see the components of the  Eq. (\ref{38}).
To make it easier to understand the rules of addition, we gave direction to the bottom edge of each element.
We can see that the rules of adding are the same as in the section 6.

\begin{figure}[H]
\begin{minipage}[t]{.45\textwidth}
\begin{center} 
\begin{tikzpicture}[>=stealth]
\draw[fill,gray!60](0,0)--(4,0)--(3.5,.8)--(.5,.8)--cycle;
\draw(.5,.8)--(0,0)--(4,0)--(3.5,.8);
\draw[thick,green](.5,.8)--(3.5,.8);
\draw[fill,green](.5,.8)circle(1pt);\draw[fill,green](3.5,.8)circle(1pt);
\draw[thick,|->](0,0)--(4,0);
\node[right] at(4,.4){$\<y+t> $ };
  
\draw[fill,gray!60](0,1)--(2,1)--(1.5,1.8)--(.5,1.8)--cycle;
\draw(.5,1.8)--(0,1)--(2,1)--(1.5,1.8);
\draw[thick,green](.5,1.8)--(1.5,1.8);
\draw[fill,green](.5,1.8)circle(1pt);\draw[fill,green](1.5,1.8)circle(1pt);
\draw[thick,|->](0,1)--(2,1);
\node[right] at(4,1.4){$\<y> $ };

\draw[fill,gray!60](2,2)--(4.5,2)--(4,2.8)--(2.5,2.8)--cycle;
\draw(2.5,2.8)--(2,2)--(4.5,2)--(4,2.8);
\draw[thick,green](2.5,2.8)--(4,2.8);
\draw[fill,green](2.5,2.8)circle(1pt);\draw[fill,green](4,2.8)circle(1pt);
\draw[thick,|->](2,2)--(4.5,2);
\node[left] at(2,2.4){$\<x+t> $ }; 

\draw[fill,red!30](4,3.8)--(4.25,3.4)--(4.5,3.8)--cycle;
\draw[thick,green](4,3.8)--(4.25,3.4)--(4.5,3.8)--cycle;
\draw[fill,green](4,3.8)circle(1pt);\draw[fill,green](4.5,3.8)circle(1pt);
\draw[fill,gray!60](4,3)--(4.5,3)--(4.25,3.4)--cycle;
\draw(4,3)--(4.5,3)--(4.25,3.4)--cycle;
\draw[thick,|->](4,3)--(4.5,3);
\draw[fill](4.25,3.4)circle(1pt);
\node[left] at(2,3.4){$\<x> $ };

\draw[fill,gray!60](2,4)--(4,4)--(3.5,4.8)--(2.5,4.8)--cycle;
\draw(2.5,4.8)--(2,4)--(4,4)--(3.5,4.8);
\draw[thick,green](2.5,4.8)--(3.5,4.8);
\draw[fill,green](2.5,4.8)circle(1pt);\draw[fill,green](3.5,4.8)circle(1pt);
\draw[thick,|->](2,4)--(4,4);
\node[left] at(2,4.4){$\<t>, \; t>0 $ };
\end{tikzpicture}
\caption{$ \<x+t>,\; \<y+t>$ for $x>0$ \hspace{1.5cm} and $y>0$.}
\label{f49}
\end{center}
\end{minipage}
\hfill
\begin{minipage}[t]{.45\textwidth}
\begin{center} 
\begin{tikzpicture}[>=stealth]
\draw[fill,gray!60](.5,0)--(3.5,0)--(3,.8)--(1,.8)--cycle;
\draw(1,.8)--(.5,0)--(3.5,0)--(3,.8);
\draw[thick,green](1,.8)--(3,.8);
\draw[fill,green](1,.8)circle(1pt);\draw[fill,green](3,.8)circle(1pt);
\draw[thick,|->](.5,0)--(3.5,0);
\node[right] at(4,.4){$\<y+t> $ };

\draw[fill,red!30](.5,1)--(1,1.8)--(0,1.8)--cycle;
\draw[thick,green](1,1.8)--(.5,1)--(0,1.8)--cycle;
\draw[fill,green](1,1.8)circle(1pt);\draw[fill,green](0,1.8)circle(1pt);
\draw[fill](.5,1)circle(1pt);
\node[right] at(4,1.4){$\<y> $ };

\draw[fill,gray!60](.5,2)--(2.5,2)--(2,2.8)--(1,2.8)--cycle;
\draw(1,2.8)--(.5,2)--(2.5,2)--(2,2.8);
\draw[thick,green](1,2.8)--(2,2.8);
\draw[fill,green](1,2.8)circle(1pt);\draw[fill,green](2,2.8)circle(1pt);
\draw[thick,|->](.5,2)--(2.5,2);
\node[right] at(4,2.4){$\<x+t> $ };

\draw[fill,red!30](2.5,3)--(3.5,3)--(4,3.8)--(2,3.8)--cycle;
\draw[red](2.5,3)--(3.5,3);
\draw[thick,green](2.5,3)--(2,3.8)--(4,3.8)--(3.5,3);
\draw[fill,green](2,3.8)circle(1pt);\draw[fill,green](4,3.8)circle(1pt);
\draw[thick,red,-](2.5,3)--(3.5,3);
\draw[thick,green,<-](2.5,3)--(2.6,3);\draw[thick,green,-|](3.45,3)--(3.5,3);
\node[right] at(4,3.4){$\<x> $ };

\draw[fill,gray!60](.5,4)--(3.5,4)--(3,4.8)--(1,4.8)--cycle;
\draw(1,4.8)--(.5,4)--(3.5,4)--(3,4.8);
\draw[thick,green](1,4.8)--(3,4.8);
\draw[fill,green](1,4.8)circle(1pt);\draw[fill,green](3,4.8)circle(1pt);
\draw[thick,|->](.5,4)--(3.5,4);
\node[right] at(4,4.4){$\<t>, \; t>0 $ };
\end{tikzpicture}
\caption{$\<x+t>, \<y+t>$  for $-t<x\leqslant 0$ and $-t<y\leqslant 0$.}
\label{f50}
\end{center}
\end{minipage}
\end{figure}

\begin{figure}[H]
\begin{minipage}[t]{.44\textwidth}
\begin{center} 
\begin{tikzpicture}[>=stealth]
\draw[fill,red!30](3.5,0)--(4,.8)--(3,.8)--cycle;
\draw[thick,green](4,.8)--(3.5,0)--(3,.8)--cycle;
\draw[fill,green](4,.8)circle(1pt);\draw[fill,green](3,.8)circle(1pt);
\draw[fill](3.5,0)circle(1pt);
\node[right] at(4,.4){$\<y+t> $ };
  
\draw[fill,red!30](1.5,1)--(3.5,1)--(4,1.8)--(1,1.8)--cycle;
\draw[thick,red,-](1.6,1)--(3.45,1);
\draw[thick,green](1.5,1)--(1,1.8)--(4,1.8)--(3.5,1);
\draw[fill,green](1,1.8)circle(1pt);\draw[fill,green](4,1.8)circle(1pt);
\draw[thick,green,<-](1.5,1)--(1.6,1);\draw[thick,green,-|](3.45,1)--(3.5,1);
\node[right] at(4,1.4){$\<y> $ };

\draw[fill,red!30](.5,2)--(1.5,2)--(2,2.8)--(0,2.8)--cycle;
\draw[thick,red,-](.6,2)--(1.45,2);
\draw[thick,green](.5,2)--(0,2.8)--(2,2.8)--(1.5,2);
\draw[fill,green](0,2.8)circle(1pt);\draw[fill,green](2,2.8)circle(1pt);
\draw[thick,green,<-](.5,2)--(.6,2);\draw[thick,green,-|](1.45,2)--(1.5,2);
\node[right] at(4,2.4){$\<x+t> $ }; 

\draw[fill,red!30](.5,3)--(3.5,3)--(4,3.8)--(0,3.8)--cycle;
\draw[thick,red,-](.6,3)--(3.45,3);
\draw[thick,green](.5,3)--(0,3.8)--(4,3.8)--(3.5,3);
\draw[fill,green](0,3.8)circle(1pt);\draw[fill,green](4,3.8)circle(1pt);
\draw[thick,green,<-](.5,3)--(.6,3);\draw[thick,green,-|](3.45,3)--(3.5,3);
\node[right] at(4,3.4){$\<x> $ };

\draw[fill,gray!60](1.5,4)--(3.5,4)--(3,4.8)--(2,4.8)--cycle;
\draw(2,4.8)--(1.5,4)--(3.5,4)--(3,4.8);
\draw[thick,green](2,4.8)--(3,4.8);
\draw[fill,green](2,4.8)circle(1pt);\draw[fill,green](3,4.8)circle(1pt);
\draw[thick,|->](1.5,4)--(3.5,4);
\node[right] at(4,4.4){$\<t>, \; t>0 $ };
\end{tikzpicture}
\caption{$ \<x+t>,\; \<y+t>$ for $x+t\leqslant 0$ \hspace{1.5cm} and $y+t\leqslant 0$.}
\label{f51}
\end{center}
\end{minipage}
\hfill
\begin{minipage}[t]{.52\textwidth}
\begin{center} 
\begin{tikzpicture}[>=stealth]
\draw[fill,red!30](2.5,0)--(4.5,0)--(5,.8)--(2,.8)--cycle;
\draw[thick,red,-](2.6,0)--(4.45,0);
\draw[thick,green](2.5,0)--(2,.8)--(5,.8)--(4.5,0);
\draw[fill,green](2,.8)circle(1pt);\draw[fill,green](5,.8)circle(1pt);
\draw[thick,green,<-](2.5,0)--(2.6,0);\draw[thick,green,-|](4.45,0)--(4.5,0);
\node[right] at(5,.4){$\<y+t> $ };
  
\draw[fill,red!30](4.5,1)--(5,1.8)--(4,1.8)--cycle;
\draw[green](5,1.8)--(4.5,1)--(4,1.8)--cycle;
\draw[fill,green](5,1.8)circle(1pt);\draw[fill,green](4,1.8)circle(1pt);
\draw[fill](4.5,1)circle(1pt);
\node[right] at(5,1.4){$\<y> $ };

\draw[fill,red!30](.5,2)--(4.5,2)--(5,2.8)--(0,2.8)--cycle;
\draw[thick,red](.6,2)--(4.45,2);
\draw[thick,green](.5,2)--(0,2.8)--(5,2.8)--(4.5,2);
\draw[fill,green](0,2.8)circle(1pt);\draw[fill,green](5,2.8)circle(1pt);
\draw[thick,green,<-](.5,2)--(.6,2);\draw[thick,green,-|](4.45,2)--(4.5,2);
\node[right] at(5,2.4){$\<x+t> $ }; 

\draw[fill,red!30](.5,3)--(2.5,3)--(3,3.8)--(0,3.8)--cycle;
\draw[thick,red](.6,3)--(2.45,3);
\draw[thick,green](.5,3)--(0,3.8)--(3,3.8)--(2.5,3);
\draw[fill,green](0,3.8)circle(1pt);\draw[fill,green](3,3.8)circle(1pt);
\draw[thick,green,<-](.5,3)--(.6,3);\draw[thick,green,-|](2.45,3)--(2.5,3);
\node[right] at(5,3.4){$\<x> $ };

\draw[fill,red!30](2.5,4)--(4.5,4)--(5,4.8)--(2,4.8)--cycle;
\draw[thick,red](2.6,4)--(4.45,4);
\draw[thick,green](2.5,4)--(2,4.8)--(5,4.8)--(4.5,4);
\draw[fill,green](2,4.8)circle(1pt);\draw[fill,green](5,4.8)circle(1pt);
\draw[thick,green,<-](2.5,4)--(2.6,4);\draw[thick,green,-|](4.45,4)--(4.5,4);
\node[right] at(5,4.4){$\<t>, \; t\leqslant 0 $ };
\end{tikzpicture}
\caption{$\<x+t>,\; \<y+t>$ for $x\leqslant 0$ \hspace{2cm} and $y\leqslant 0$.}
\label{f52}
\end{center}
\end{minipage}
\end{figure}

\begin{figure}[H]
\begin{minipage}[t]{.38\textwidth}
\begin{center} 
\begin{tikzpicture}[>=stealth]
\draw[fill,red!30](.5,0)--(1,.8)--(0,.8)--cycle;
\draw[thick,green](0,.8)--(.5,0)--(1,.8)--cycle;
\draw[fill,green](0,.8)circle(1pt);\draw[fill,green](1,.8)circle(1pt);
\draw[fill](.5,0)circle(1pt);
\node[right] at(3,.4){$\<y+t> $ };
  
\draw[fill,gray!60](.5,1)--(2.5,1)--(2,1.8)--(1,1.8)--cycle;
\draw(1,1.8)--(.5,1)--(2.5,1)--(2,1.8);
\draw[thick,green](1,1.8)--(2,1.8);
\draw[fill,green](1,1.8)circle(1pt);\draw[fill,green](2,1.8)circle(1pt);
\draw[thick,|->](.5,1)--(2.5,1);
\node[right] at(3,1.4){$\<y> $ };

\draw[fill,red!30](1.5,2)--(2.5,2)--(3,2.8)--(1,2.8)--cycle;
\draw[thick,red](1.6,2)--(2.45,2);
\draw[thick,green](1.5,2)--(1,2.8)--(3,2.8)--(2.5,2);
\draw[fill,green](1,2.8)circle(1pt);\draw[fill,green](3,2.8)circle(1pt);
\draw[thick,green,<-](1.5,2)--(1.6,2);\draw[thick,green,-|](2.45,2)--(2.5,2);
\node[right] at(3,2.4){$\<x+t> $ }; 

\draw[fill,gray!60](.5,3)--(1.5,3)--(1,3.8)--cycle;
\draw(.5,3)--(1.5,3)--(1,3.8)--cycle;
\draw[fill,green](1,3.8)circle(1pt);
\draw[thick,|->](.5,3)--(1.5,3);
\node[right] at(3,3.4){$\<x> $ };

\draw[fill,red!30](.5,4)--(2.5,4)--(3,4.8)--(0,4.8)--cycle;
\draw[thick,red](.6,4)--(2.45,4);
\draw[thick,green](.5,4)--(0,4.8)--(3,4.8)--(2.5,4);
\draw[fill,green](0,4.8)circle(1pt);\draw[fill,green](3,4.8)circle(1pt);
\draw[thick,green,<-](.5,4)--(.6,4);\draw[thick,green,-|](2.45,4)--(2.5,4);
\node[right] at(3,4.4){$\<t>, \; t\leqslant 0 $ };
\end{tikzpicture}
\caption{$ \<x+t>,\; \<y+t>$ for \hspace{1.5cm} $-t\geqslant x> 0$  and $-t\geqslant y> 0$.}
\label{f53}
\end{center}
\end{minipage}
\hfill
\begin{minipage}[t]{.52\textwidth}
\begin{center} 
\begin{tikzpicture}[>=stealth]
\draw[fill,gray!60](0,0)--(2,0)--(1.5,.8)--(.5,.8)--cycle;
\draw(.5,.8)--(0,0)--(2,0)--(1.5,.8);
\draw[thick,green](.5,.8)--(1.5,.8);
\draw[fill,green](.5,.8)circle(1pt);\draw[fill,green](1.5,.8)circle(1pt);
\draw[thick,|->](0,0)--(2,0);
\node[right] at(5,.4){$\<y+t> $ };
  
\draw[fill,gray!60](0,1)--(4,1)--(3.5,1.8)--(.5,1.8)--cycle;
\draw(.5,1.8)--(0,1)--(4,1)--(3.5,1.8);
\draw[thick,green](.5,1.8)--(3.5,1.8);
\draw[fill,green](.5,1.8)circle(1pt);\draw[fill,green](3.5,1.8)circle(1pt);
\draw[thick,|->](0,1)--(4,1);
\node[right] at(5,1.4){$\<y> $ };

\draw[fill,red!30](4,2.8)--(4.25,2.4)--(4.5,2.8)--cycle;
\draw[thick,green](4,2.8)--(4.25,2.4)--(4.5,2.8)--cycle;
\draw[fill,green](4,2.8)circle(1pt);\draw[fill,green](4.5,2.8)circle(1pt);
\draw[fill,gray!60](4,2)--(4.5,2)--(4.25,2.4)--cycle;
\draw(4,2)--(4.5,2)--(4.25,2.4)--cycle;
\draw[thick,|->](4,2)--(4.5,2);
\draw[fill](4.25,2.4)circle(1pt);
\node[right] at(5,2.4){$\<x+t> $ }; 

\draw[fill,gray!60](2,3)--(4.5,3)--(4,3.8)--(2.5,3.8)--cycle;
\draw(2.5,3.8)--(2,3)--(4.5,3)--(4,3.8);
\draw[thick,green](2.5,3.8)--(4,3.8);
\draw[fill,green](2.5,3.8)circle(1pt);\draw[fill,green](4,3.8)circle(1pt);
\draw[thick,|->](2,3)--(4.5,3);
\node[right] at(4.5,3.4){$\<x> $ };

\draw[fill,red!30](2,4)--(4,4)--(4.5,4.8)--(1.5,4.8)--cycle;
\draw[thick,red](2.1,4)--(3.95,4);
\draw[thick,green](2,4)--(1.5,4.8)--(4.5,4.8)--(4,4);
\draw[fill,green](1.5,4.8)circle(1pt);\draw[fill,green](4.5,4.8)circle(1pt);
\draw[thick,green,<-](2,4)--(2.1,4);\draw[thick,green,-|](3.95,4)--(4,4);
\node[right] at(5,4.4){$\<t>, \; t\leqslant 0 $ };
\end{tikzpicture}
\caption{$\<x+t>,\; \<y+t>$ for $x+t> 0$ \hspace{1.5cm} and $y+t> 0$.}
\label{f54}
\end{center}
\end{minipage}
\end{figure}

\noindent  We must show that the Eq. (\ref{382}) is true in the geometric sense for all $x,y,t$.

 \begin{thm}
      $\forall x,y,t \in \mathbb{R}$ Equation (\ref{382}) is true in the geometric sense. \label{t7.1}
\end{thm}
\begin{proof}
\noindent  We need to consider three cases from section 6 for positive  $ x,y $ and any $ t $.\\
\begin{enumerate}
\item[(1)] $t >0$, (Fig. ~\ref{f55}).\\[.1cm]
\text{In next cases  $t<0$}.
\item[(2)] $x+t>0, y+t>0 $, (Fig. ~\ref{f56}).
\item[(3)] $x+t>0, y+t<0$, (Fig. ~\ref{f57}).
\end{enumerate}
\begin{figure}[H]
\begin{minipage}[t]{.44\textwidth}
\begin{center} 
\begin{tikzpicture}[>=stealth]
\draw[fill,gray!60](0,0)--(4,0)--(3.5,.8)--(.5,.8)--cycle;
\draw(.5,.8)--(0,0)--(4,0)--(3.5,.8);
\draw[thick,green](.5,.8)--(3.5,.8);
\draw[fill,green](.5,.8)circle(1pt);\draw[fill,green](3.5,.8)circle(1pt);
\draw[thick,|->](0,0)--(4,0);
\node[right] at(4,.4){$\<x+y+t> $ };
  
\draw[fill,red!30](1,1)--(2,1)--(1.5,1.8)--cycle;
\draw[red](1,1)--(2,1)--(1.5,1.8)--cycle;
\draw[red,thick,<-|](1,1)--(2,1);
\draw[fill,green](1.5,1.8)circle(1pt);
\node[right] at(4,1.4){$-\<t> $ };

\draw[fill,gray!60](0,2)--(2,2)--(1.5,2.8)--(0.5,2.8)--cycle;
\draw(.5,2.8)--(0,2)--(2,2)--(1.5,2.8);
\draw[thick,green](.5,2.8)--(1.5,2.8);
\draw[fill,green](.5,2.8)circle(1pt);\draw[fill,green](1.5,2.8)circle(1pt);
\draw[thick,|->](0,2)--(2,2);
\node[right] at(4,2.4){$\<y+t> $ }; 

\draw[fill,gray!60](1,3)--(4,3)--(3.5,3.8)--(1.5,3.8)--cycle;
\draw(1.5,3.8)--(1,3)--(4,3)--(3.5,3.8);
\draw[thick,green](1.5,3.8)--(3.5,3.8);
\draw[fill,green](1.5,3.8)circle(1pt);\draw[fill,green](3.5,3.8)circle(1pt);
\draw[thick,|->](1,3)--(4,3);
\node[right] at(4,3.4){$\<x+t> $ };

\draw[fill,gray!60](1,4)--(2,4)--(1.5,4.8)--cycle;
\draw(1,4)--(2,4)--(1.5,4.8)--cycle;
\draw[thick,|->](1,4)--(2,4);
\draw[fill,green](1.5,4.8)circle(1pt);
\node[right] at(4,4.4){$\<t> $ };
\end{tikzpicture}
\caption{$ \<x+y+t>,$ for $x>0$,\hspace{.8cm} $y>0,\ t>0$.}
\label{f55}
\end{center}
\end{minipage}
\hfill
\begin{minipage}[t]{.45\textwidth}
\begin{center} 
\begin{tikzpicture}[>=stealth]
\draw[fill,gray!60](0,0)--(3.5,0)--(3,.8)--(.5,.8)--cycle;
\draw(.5,.8)--(0,0)--(3.5,0)--(3,.8);
\draw[thick,green](.5,.8)--(3,.8);
\draw[fill,green](.5,.8)circle(1pt);\draw[fill,green](3,.8)circle(1pt);
\draw[thick,|->](0,0)--(3.5,0);
\node[right] at(3.5,.4){$\<x+y+t> $ };
  
\draw[fill,gray!60](2,1)--(3,1)--(3.5,1.8)--(1.5,1.8)--cycle;
\draw[thick,green](2,1)--(1.5,1.8)--(3.5,1.8)--(3,1);
\draw[fill,green](1.5,1.8)circle(1pt);\draw[fill,green](3.5,1.8)circle(1pt);
\draw[thick,green,|-](2,1)--(2.05,1);\draw[thick,green,->](2.9,1)--(3,1);
\draw[thick](2.05,1)--(2.9,1);
\node[right] at(3.5,1.4){$-\<t> $ };

\draw[fill,gray!60](0,2)--(2,2)--(1.5,2.8)--(.5,2.8)--cycle;
\draw(.5,2.8)--(0,2)--(2,2)--(1.5,2.8);
\draw[thick,green](.5,2.8)--(1.5,2.8);
\draw[fill,green](.5,2.8)circle(1pt);\draw[fill,green](1.5,2.8)circle(1pt);
\draw[thick,|->](0,2)--(2,2);
\node[right] at(3.5,2.4){$\<y+t> $ }; 

\draw[fill,red!30](3,3.8)--(3.25,3.4)--(3.5,3.8)--cycle;
\draw[thick,green](3,3.8)--(3.25,3.4)--(3.5,3.8)--cycle;
\draw[fill,green](3,3.8)circle(1pt);\draw[fill,green](3.5,3.8)circle(1pt);
\draw[fill,gray!60](3,3)--(3.5,3)--(3.25,3.4)--cycle;
\draw(3,3)--(3.5,3)--(3.25,3.4)--cycle;
\draw[thick,|->](3,3)--(3.5,3);
\draw[fill](3.25,3.4)circle(1pt);
\node[right] at(3.5,3.4){$\<x+t> $ };

\draw[fill,red!30](2,4)--(3,4)--(3.5,4.8)--(1.5,4.8)--cycle;
\draw[thick,green](3,4)--(3.5,4.8)--(1.5,4.8)--(2,4)--cycle;
\draw[fill,green](3.5,4.8)circle(1pt);\draw[fill,green](1.5,4.8)circle(1pt);
\draw[thick,green,<-](2,4)--(2.1,4);\draw[thick,green,-|](2.95,4)--(3,4);
\draw[thick,red](2.1,4)--(2.95,4);
\node[right] at(3.5,4.4){$\<t> $ };
\end{tikzpicture}
\caption{$ \<x+y+t>,$ for $x+t>0$,\hspace{.5cm} $y+t>0$, $t<0$.}
\label{f56}
\end{center}
\end{minipage}
\end{figure}

 \begin{figure}[H]
\begin{center} 
\begin{tikzpicture}[>=stealth]
\draw[fill,gray!60](1.5,0)--(4.5,0)--(4,.8)--(2,.8)--cycle;
\draw(2,.8)--(1.5,0)--(4.5,0)--(4,.8);
\draw[thick,green](2,.8)--(4,.8);
\draw[fill,green](2,.8)circle(1pt);\draw[fill,green](4,.8)circle(1pt);
\draw[thick,|->](1.5,0)--(4.5,0);
\node[right] at(4.5,.4){$\<x+y+t> $ };
  
\draw[fill,gray!60](.5,1)--(2.5,1)--(3,1.8)--(0,1.8)--cycle;
\draw[thick,green](2.5,1)--(3,1.8)--(0,1.8)--(.5,1);
\draw[fill,green](0,1.8)circle(1pt);\draw[fill,green](3,1.8)circle(1pt);
\draw[thick,green,|-](.5,1)--(.55,1);\draw[thick,green,->](2.4,1)--(2.5,1);
\draw[thick](.55,1)--(2.4,1);
\node[right] at(4.5,1.4){$-\<t> $ };

\draw[fill,red!30](.5,2)--(1.5,2)--(2,2.8)--(0,2.8)--cycle;
\draw[thick,green](1.5,2)--(2,2.8)--(0,2.8)--(.5,2);
\draw[fill,green](0,2.8)circle(1pt);\draw[fill,green](2,2.8)circle(1pt);
\draw[thick,green,<-](.5,2)--(.6,2);\draw[thick,green,-|](1.45,2)--(1.5,2);
\draw[thick,red](.6,2)--(1.45,2);
\node[right] at(4.5,2.4){$\<y+t> $ }; 

\draw[fill,gray!60](2.5,3)--(4.5,3)--(4,3.8)--(3,3.8)--cycle;
\draw(4,3.8)--(4.5,3)--(2.5,3)--(3,3.8);
\draw[thick,green](4,3.8)--(3,3.8);
\draw[fill,green](4,3.8)circle(1pt);\draw[fill,green](3,3.8)circle(1pt);
\draw[thick,|->](2.5,3)--(4.5,3);
\node[right] at(4.5,3.4){$\<x+t> $ };

\draw[fill,red!30](.5,4)--(2.5,4)--(3,4.8)--(0,4.8)--cycle;
\draw[thick,green](2.5,4)--(3,4.8)--(0,4.8)--(.5,4);
\draw[fill,green](3,4.8)circle(1pt);\draw[fill,green](0,4.8)circle(1pt);
\draw[thick,green,<-](.5,4)--(.6,4);\draw[thick,green,-|](2.45,4)--(2.5,4);
\draw[thick,red](.6,4)--(2.45,4);
\node[right] at(4.5,4.4){$\<t> $ };
\end{tikzpicture}
\caption{$ \<x+y+t>,$ for $x+t>0$, $y+t<0, t<0$.}
\label{f57}
\end{center}
\end{figure}

\noindent The cases for $x=0 \vee y=0 \vee t=0$ we leave for consideration to the reader.\\
\end{proof}

\noindent Using Eq.(\ref{38})  we can check  the interpretation of the element $\<\frac{1}{3}+\frac{1}{3}+0>=\<\frac{2}{3}>$
 (Fig.~\ref{f58}). (The Figure is to a scale enlarged twice).
 \begin{figure}[H]
\begin{center} 
\begin{tikzpicture}[scale=1.36]
\draw[fill,red!30](.5,.4)--(1,1.2)--(0,1.2)--cycle;
\draw[thick,green](1,1.2)--(0,1.2);
\draw[thick,green](1,1.2)--(.5,.4)--(0,1.2);
\draw[fill,green](0,1.2)circle(1pt);
\draw[fill,gray!60](.25,0)--(.75,0)--(.5,.4)--cycle;
\draw(.25,0)--(.75,0)--(.5,.4)--cycle;
\draw[fill](.5,.4)circle(1pt);

\draw[fill,red!30](1,.4)--(1.5,1.2)--(.5,1.2)--cycle;
\draw[thick, green](1.5,1.2)--(1,.4)--(.75,.8);
\draw[red](.5,1.2)--(.75,.8)--(1,1.2);
\draw[thick,green](1.5,1.2)--(0,1.2);
\draw[fill,green](.75,.8)circle(1pt);

\draw[fill,green](1.5,1.2)circle(1pt);
\draw[fill,gray!60](.75,0)--(1.25,0)--(1,.4)--cycle;
\draw(.75,0)--(1.25,0)--(1,.4)--cycle;
\draw[fill](1,.4)circle(1pt);
\draw[fill](.75,0)circle(1pt);
\draw node at(.75,1.06){\small{2}};
\draw node at(.75,-.2){\small{2}};
\node at(1.75,.6){+};

\draw[fill,gray!60](2,1.2)--(2.75,0)--(3.5,1.2)--cycle;
\draw[thick,green](2,1.2)--(2.75,0)--(3.5,1.2)--cycle;
\draw[fill,green](2,1.2)circle(1pt);\draw[fill,green](3.5,1.2)circle(1pt);
\draw[fill,red](2.75,0)circle(1pt);
\node at(3.75,.6){=};

\draw[fill,red!30](4.25,1.2)--(4.5,.8)--(4.75,1.2)--cycle;
\draw[thick,green](4.25,1.2)--(4.5,.8)--(4.75,1.2)--cycle;
\draw[fill,green](4.25,1.2)circle(1pt);\draw[fill,green](4.75,1.2)circle(1pt);
\draw[fill](4.5,.8)circle(1pt);
\draw[fill,gray!60](4,0)--(5,0)--(4.5,.8)--cycle;
\draw(4,0)--(5,0)--(4.5,.8)--cycle;
\end{tikzpicture}
\caption{$\<\frac{1}{3}>+\<\frac{1}{3}>+\big(-\<0>\big)=\<\frac{2}{3}>$. }
\label{f58}
\end{center}
\end{figure} 

\noindent The Figure ~\ref{f59} shows that the interpretation of the element $\<\frac{1}{3}>$ is well defined.

\begin{figure}[H]
\begin{center} 
\begin{tikzpicture}[scale=1.36]
\draw[fill,red!30](.5,.4)--(1,1.2)--(0,1.2)--cycle;
\draw[thick,green](1,1.2)--(.5,.4)--(0,1.2)--cycle;
\draw[fill,green](0,1.2)circle(1pt);\draw[fill,green](1,1.2)circle(1pt);
\draw[fill,gray!60](.25,0)--(.75,0)--(.5,.4)--cycle;
\draw(.25,0)--(.75,0)--(.5,.4)--cycle;
\draw[fill](.5,.4)circle(1pt);
\node at(.9,.6){+};

\draw[fill,red!30](1,1.2)--(1.25,.8)--(1.5,1.2)--cycle;
\draw[thick, green](1,1.2)--(1.25,.8)--(1.5,1.2)--cycle;
\draw[fill,green](1,1.2)circle(1pt);\draw[fill,green](1.5,1.2)circle(1pt);
\draw[fill,gray!60](.75,0)--(1.75,0)--(1.25,.8)--cycle;
\draw(.75,0)--(1.75,0)--(1.25,.8)--cycle;
\draw[fill](1.25,.8)circle(1pt);
\draw[fill](.75,0)circle(1pt);
\draw node at(.75,-.2){\small{2}};
\node at(2,.6){+};

\draw[fill,gray!60](2.25,1.2)--(3,0)--(3.75,1.2)--cycle;
\draw[thick,green](2.25,1.2)--(3,0)--(3.75,1.2)--cycle;
\draw[fill,green](2.25,1.2)circle(1pt);\draw[fill,green](3.75,1.2)circle(1pt);
\draw[fill,red](3,0)circle(1pt);
\node at(4,.6){=};

\draw[fill,gray!60](4.25,0)--(5.75,0)--(5,1.2)--cycle;
\draw[](4.25,0)--(5.75,0)--(5,1.2)--cycle;
\draw[fill,green](5,1.2)circle(1pt);
\end{tikzpicture}
\caption{$\<\frac{1}{3}>+\<\frac{2}{3}>+\big(-\<0>\big)=\<1>$. }
\label{f59}
\end{center}
\end{figure}

\noindent Let us assume that $\<1>_1=\<1>_2-\<0>_2$ , where $\<x>_1=\<x> \in \mathbb{R}_1 $, \;
$\<x>_2=\<x> \in \mathbb{R}_2 $. \\
Then\\ 
$\forall x\in \mathbb{R} \;\; \<x>_1 = \<x>_2-\<x-1>_2$.\\
But only in the geometric sense because\\
$\<x>_1 =(x,1)$ and $\<x>_2-\<x-1>=(2x-1,1,0)$.\\

\noindent The above interpretation of the set $\mathbb{R}_1$ is not the only one.
We can  use, for example the following interpretation.\\
\tikz{\draw[fill,gray!60] (0,0)--(2,0)--(1.5,.8)--(.5,.8)--cycle;
\draw (0,0)--(2,0)--(1.5,.8)--(.5,.8)--cycle;} \ as the  element $\<1>$ and \quad
\tikz{\draw[fill,gray!70] (0,0)--(1,0)--(.5,.8)--(0,0);
\draw (0,0)--(1,0)--(.5,.8)--(0,0);
\draw[fill](0,0)circle(1pt);\draw[fill](1,0)circle(1pt);
\draw[fill](.5,.8)circle(1pt);} \ as the element $\<0>$.\\

\noindent Geometric interpretation of the set $\mathbb{R}_1$ in the  section 7 shows that the elements of  
  the set $\mathbb{R}_2$ can be interpreted in the form of three-dimensional figures.
 But that is material for another article.


\begin{thebibliography}{0}
\bibitem {Wol} Wolfram MathWorld. 

\bibitem {Dui} A.J.W Duijvestijn, {\it Simple perfect squared square of lowest order} J. Combin. Theory, Ser. B, 25, pp. 240-243, (1978)
\bibitem {Gam} I. Gambini, {\it A method for cutting squares into distinct squares}, Discrete Applied Mathematics, Volume 98, Issues 1–2,  pp. 65-80, (1999)
\bibitem{Tut1} W. T. Tutte, {\it Dissections into Equilateral Triangles}, The Mathematical Gardner, edited by David Klarner. Published by Van Nostrand Reinhold (1981).
\bibitem{Tut2 } W. T. Tutte, {\it The Dissection of Equilateral Triangles into Equilateral Triangles},  Proc. Cambridge Phil. Soc 44,  pp. 463-482, (1948)
\bibitem{DraHam} A. Dr\'apal, C. H\"am\"al\"ainen,  {\it An enumeration of equilateral triangle dissections}, Discrete Applied Mathematics, Volume 158, Issue 14 (2010)

\bibitem{mathworld}
Wolfram, E. \textit{Homothety}. Wolfram MathWorld.

\bibitem{artin}
Artin, M. \textit{Algebra}. Pearson, 2011.

\bibitem{coxeter}
Coxeter, H. S. M. \textit{Introduction to Geometry}. Wiley, 1969.
\end{thebibliography}
\end{document}